\documentclass[11pt,leqno]{amsart}

\usepackage[top=2.5 cm,bottom=2 cm,left=2.5 cm,right=2 cm]{geometry}
\usepackage[utf8]{inputenc}

\usepackage{bbm}
\usepackage{esint}
\usepackage{graphicx}
\usepackage{hyperref}
\hypersetup{
    hidelinks,
    }
\usepackage{color}
\usepackage{amssymb}
\usepackage{amsfonts}
\usepackage{psfrag}
\newcounter{stepnb}

\usepackage{tikz}
\usetikzlibrary{shapes.geometric,calc,decorations.markings,math}

\tikzstyle{nodo}=[circle,draw,fill,inner sep=0pt,minimum size=%
1.5mm]
\tikzstyle{infinito}=[circle,inner sep=0pt,minimum size=0mm]

\newtheorem{theorem}{Theorem}[section]
\newtheorem{lemma}[theorem]{Lemma}
\newtheorem{proposition}[theorem]{Proposition}
\newtheorem{corol}[theorem]{Corollary}

\numberwithin{equation}{section}

\newcommand{\R}{\mathbb{R}}

\DeclareMathOperator*{\essup}{ess\,sup}

\newcommand{\ee}{\varepsilon}

\newcommand{\be}{\begin{equation}}
\newcommand{\eq}{\end{equation}}

\newcommand{\weaks}{\stackrel{*}{\rightharpoonup}}

\begin{document}
\title[Nonlocal traffic models with general kernels]{Nonlocal traffic models with general kernels: \\singular limit, entropy admissibility, and convergence rate
}

\author[M.~Colombo]{Maria Colombo}
\address{M.C. EPFL B, Station 8, CH-1015 Lausanne, Switzerland.}
\email{maria.colombo@epfl.ch}
\author[G.~Crippa]{Gianluca Crippa}
\address{G.C. Departement Mathematik und Informatik,
Universit\"at Basel, Spiegelgasse 1, CH-4051 Basel, Switzerland.}
\email{gianluca.crippa@unibas.ch}
\author[E. Marconi]{Elio Marconi}
\address{E.M. EPFL B, Station 8, CH-1015 Lausanne, Switzerland.}
\email{elio.marconi@epfl.ch}
\author[L.~V.~Spinolo]{Laura V.~Spinolo}
\address{L.V.S. IMATI-CNR, via Ferrata 5, I-27100 Pavia, Italy.}
\email{spinolo@imati.cnr.it}
\maketitle
{
\rightskip .85 cm
\leftskip .85 cm
\parindent 0 pt
\begin{footnotesize}
Nonlocal conservation laws (the signature feature being that the flux function depends on the solution through the convolution with a given kernel) are extensively used in the modeling of vehicular traffic. In this work we discuss the singular local limit, namely the convergence of the nonlocal solutions to the entropy admissible solution of the conservation 
law obtained by replacing the convolution kernel with a Dirac delta. 
Albeit recent counter-examples rule out convergence in the general case, in the specific framework of traffic models (with anisotropic convolution kernels) the singular limit has been established under rigid assumptions, i.e. {in the case of the exponential kernel (which entails algebraic identities between the kernel and its derivatives) or under fairly restrictive requirements on the initial datum}.
In this work we obtain general convergence results under assumptions that are entirely natural in view of applications to traffic models, plus a convexity requirement on the convolution kernels. {We then provide a general criterion for entropy admissibility of the limit and a convergence rate. We {also} exhibit a counter-example showing that the convexity assumption is necessary for our main compactness estimate. }

\medskip



\end{footnotesize}
}

\section{Introduction}
Consider a family of Cauchy problems for nonlocal conservation laws in the form 
\be \label{e:nlc}
     \left\{
     \begin{array}{ll}
     \partial_t u_\ee + \partial_x [ V(u_\ee \ast \eta_\ee) u_\ee ] = 0 \\
     u_\ee (0, \cdot) = u_0 \\
     \end{array}
     \right. 
     \qquad 
     \text{with $\eta_\ee (x): = \frac{1}{\ee} \eta \left( \frac{x}{\ee}\right).$} 
\eq
In the previous expression, $u_\ee: \R_+ \times \R \to \R$ is the unknown, $\eta: \R \to \R_+$ is a given convolution kernel, the symbol $\ast$ represents the convolution with respect to the $x$ variable only,  $\ee >0$ is a parameter and $V: \R \to \R$ a Lipschitz continuous function. We specify in the following the precise assumptions imposed on $V$, $\eta$ and $u_0$, for the time being we just mention that existence and uniqueness results for~\eqref{e:nlc} can be established under fairly general assumptions, see for instance~\cite{KeimerPflug}. In the present work we are concerned with the singular local limit $\ee \to 0+$: when $\eta_\ee$ converges weakly$^\ast$ in the sense of measures to the Dirac delta, \eqref{e:nlc} formally boils down to the conservation law Cauchy problem
\be \label{e:cl}
       \left\{
     \begin{array}{ll}
     \partial_t u + \partial_x [ V(u) u] = 0 \\
     u (0, \cdot) = u_0 . \\
     \end{array}
     \right. 
\eq
Existence and uniqueness results for so-called entropy admissible solutions of~\eqref{e:cl} date back to Kru{\v{z}}kov~\cite{Kruzkov}. 
{The nonlocal-to-local limit was first addressed by Zumbrun in~\cite{Zumbrun} and Amorim, R. Colombo and Teixeira in~\cite{ACT}}. Zumbrun~\cite{Zumbrun} showed that when $\ee \to 0^+$ the family $u_\ee$ converges to the entropy admissible solution $u$ of~\eqref{e:cl} provided $V(u)=u$, $\eta$ is  an even function and the limit solution  $u$ is regular.  In~\cite{ACT} the authors posed the general convergence question and exhibited numerical experiments suggesting convergence of $u_\ee$ to the entropy admissible solution $u$.  In~\cite{ColomboCrippaSpinolo} the authors provided some counter-examples showing that, in general, the family  $u_\ee$ does not converge to the entropy admissible solution of~\eqref{e:cl}.

The analysis in~\cite{ColomboCrippaSpinolo} left open the possibility that convergence holds under specific assumptions. In this respect, a natural target for the investigation of the local limit is the framework of traffic flow models: indeed, in recent years, nonlocal conservation laws in the form~\eqref{e:nlc} have been widely used in the modeling of traffic, see for instance~\cite{BlandinGoatin,ChiarelloGoatin,ColomboGaravelloMercier} and the references therein. In this framework, the local counterpart~\eqref{e:cl} is the by now classical LWR model introduced in~\cite{LW,R}: $u$ in this case represents the density of cars and $V$ their speed. It is then natural to assume that the initial datum $u_0$ satisfies 
\be \label{e:u0}
         u_0 \in L^\infty(\R), \qquad 0 \leq u_0 \leq 1,
\eq
where we have normalized to $1$ the maximum possible car density, corresponding to bumper-to-bumper packing. Also, note that the LWR model postulates that drivers 
choose their speed depending on the car density, and, since the expected reaction to a traffic congestion is deceleration, the standard 
assumptions imposed on the function $V$ are   
\be \label{e:V}
    V \in \mathrm{Lip}(\R), \qquad V' \leq 0 \; \text{on [0, 1]}.
\eq
{The nonlocal convolution term in~\eqref{e:nlc} models the fact that drivers choose their speed based on the density of cars in a suitable neighborhood. }
 The standard hypotheses imposed on $\eta$ are then 
\be \label{e:eta}
      \eta \in L^1(\R) \cap L^\infty (\R), \quad \mathrm{supp} \, \eta \subseteq \R_-, \qquad \eta \ge 0, \qquad \eta \text{ non-decreasing on $\R_-$}, \qquad \int_{\R_-} \eta(x) dx =1. 
\eq
The assumption $\mathrm{supp} \, \eta \subseteq \R_-$ models the fact that drivers choose their speed based on the downstream car density only (i.e. they look forward, not backward), whereas the monotonicity condition takes into account the fact  that they pay more attention to closer vehicles. 

We now focus on the analysis of nonlocal Cauchy problems satisfying~\eqref{e:u0},\eqref{e:V} and~\eqref{e:eta}. 
Under~\eqref{e:V} and~\eqref{e:eta} nonlocal conservation laws~\eqref{e:nlc} enjoy {the maximum principle and, under some more assumptions, propagation of monotonicity}, see~\cite{BlandinGoatin}. Keimer and Pflug~\cite{KeimerPflug2} used these properties to establish convergence in the local limit $\ee \to 0^+$ for monotone initial data.  Bressan and Shen~\cite{BressanShen,BressanShen2} established convergence in the local limit in the case $\eta(x) =  \mathbbm{1}_{]-\infty, 0]}(x) e^x$ and  under the assumption that the initial datum has finite total variation and is bounded away from $0$. A key point of the analysis in both~\cite{BressanShen,KeimerPflug2} is that the total variation $\mathrm{TotVar} \ u_\ee (t, \cdot)$ is a monotone non-increasing function of time. Note furthermore that the assumption that the initial datum has bounded total variation is fairly natural in the conservation laws framework, see~\cite{Dafermos:book}. 
In our previous work~\cite{ColomboCrippaMarconiSpinolo} we established convergence via an Ole\u{\i}nik-type estimate under quite general assumptions on the kernel $\eta$, but requiring that the initial datum $u_0$ is bounded away from $0$ and satisfies a fairly  restrictive one-sided Lipschitz condition
. In~\cite{ColomboCrippaMarconiSpinolo} we also exhibit a counter-example showing that, if the initial datum attains the value $0$ then it may happen that for every $t>0$ 
the total variation $\mathrm{TotVar} \ u_\ee (t, \cdot)$ blows up as $\ee \to 0^+$: this rules out total variation estimates and implies that the convergence proofs in~\cite{BressanShen,ColomboCrippaMarconiSpinolo} cannot extend to general initial data. A way out this obstruction has been recently found by Coclite, Coron, De Nitti, Keimer and Pflug in~\cite{CocliteCoronDNKeimerPflug}: rather than looking at the total variation of $u_\ee$, they show that  the total variation of the convolution term  
\be  \label{e:w}
      w_\ee (t, x) : = u_\ee \ast \eta_\ee (t, x) = \frac{1}{\ee} \int_{x}^{+\infty} \eta \left( \frac{x-y}{\ee} \right) u_\ee (t, y) dy 
\eq
is a monotone non-increasing function of time. This allows to extend the results by Bressan and Shen~\cite{BressanShen,BressanShen2} to initial data attaining the value $0$. Note, however, that the analysis in~\cite{BressanShen,BressanShen2,CocliteCoronDNKeimerPflug} is restricted to the case of the exponential kernel $\eta(x) =  \mathbbm{1}_{]-\infty, 0]}(x) e^x$. The proofs crucially rely on the algebraic identity, with no analogue for general kernels,
\begin{equation}
\label{eqn:algebra}u_\ee = w_\ee - \ee \partial_x w_\ee \qquad \mbox{if } \eta(x) =  \mathbbm{1}_{]-\infty, 0]}(x) e^x, 
\end{equation}
which is used to derive an equation for $w_\ee$ independent of $u_\ee$ and allows to reformulate~\eqref{e:nlc} as a system with relaxation.

Our first main result states that the fact that $\mathrm{TotVar} \ w_\ee (t, \cdot)$ is a non-increasing function of time is true under the sole minimal assumptions~\eqref{e:V} and~\eqref{e:eta} combined with a convexity requirement, namely
\begin{equation}
\label{eqn:convex}
\eta \mbox{ is convex on }  \R_-,
\end{equation}
 that we comment upon in the following. 
\begin{theorem} \label{t:tvb} 
Assume that $u_0$, $V$ and $\eta$ satisfy~\eqref{e:u0},~\eqref{e:V}, \eqref{e:eta} and~\eqref{eqn:convex}, respectively, and let $u_\ee$ be the solution of the Cauchy problem~\eqref{e:nlc} and $w_\ee$  be as in~\eqref{e:w}. If $\mathrm{TotVar} \ u_0 < + \infty$ then 
\be \label{e:tv}
     \mathrm{TotVar} \, w_\ee(t, \cdot) \leq \mathrm{TotVar} \, w_\ee(0, \cdot)  \quad \text{for every $\ee>0$ and a.e. $t>0$}. 
\eq
\end{theorem}
Note that the right-hand side of~\eqref{e:tv} is then easily controlled by the inequality
\be \label{e:tv0}
     \mathrm{TotVar} \ w_\ee (0, \cdot) \leq \mathrm{TotVar} \ u_0,
\eq 
which directly follows from~\eqref{e:w}. 
The proof of Theorem~\ref{t:tvb} is  independent from the one in~\cite{CocliteCoronDNKeimerPflug} since it does not rely on any variant of \eqref{eqn:algebra}, and more delicate since the equation for $w_\ee$ is more involved, see~\eqref{e:ewgen}. Note furthermore that, by relying on~\eqref{e:tv}, it is rather easy (see \S\ref{ss:cp}) to see that, up to subsequences,  
\be \label{e:vali}
      {w_\ee \to u \;\; \text{strongly in $L^1_{\mathrm{loc}}(\R_+ \times \R)$,}} \qquad 
      u_\ee \weaks u \quad \text{weakly$^\ast$ in $L^\infty (\R_+ \times \R)$, }
\eq
for some function $u$ that is a \emph{distributional} solution of the Cauchy problem~\eqref{e:cl}. 
The problem of the entropy admissibility of the limit is {highly} nontrivial and, even in the case of the exponential kernel $\eta(x) =  \mathbbm{1}_{]-\infty, 0]}(x) e^x$ it was  partially left open in~\cite{BressanShen} and treated specifically in~\cite{BressanShen2,CocliteCoronDNKeimerPflug} by relying on formula \eqref{eqn:algebra}. In~\cite{ColomboCrippaMarconiSpinolo} the entropy admissibility follows from an Ole\u{\i}nik-type estimate~\cite{Oleinik} whose proof requires the one-sided Lipschitz condition on the initial datum. Our second main result establishes entropy admissibility of the limit function under very minimal assumptions. 
\begin{theorem} \label{t:entropy}
Assume that $u_0$, $V$ and $\eta$ satisfy~\eqref{e:u0}~\eqref{e:V} and~\eqref{e:eta}, respectively, and let  $w_\ee$  be as in~\eqref{e:w}, where $u_\ee$ is the solution of the Cauchy problem~\eqref{e:nlc}. Consider a sequence $\ee_k \to 0^+$ and assume that  {$w_{\ee_k} \to u$ in $L^1_{\mathrm{loc}}(\R_+ \times \R)$ for some function $u \in L^\infty (\R_+ \times \R)$; } then $u$ is the entropy admissible solution of~\eqref{e:cl}. 
\end{theorem} 
We explicitly point out that the convexity assumption~\eqref{eqn:convex} is not needed in the statement of Theorem~\ref{t:entropy}. Also, Theorem~\ref{t:entropy} only requires strong $L^1$ compactness rather than total variation bounds, and as such it provides a general entropy admissibility criterion for the limit, which replaces the ad hoc arguments used in specific cases in~\cite{BressanShen2,CocliteCoronDNKeimerPflug,KeimerPflug2} and might be useful in contexts where the $L^1$ compactness is obtained through weaker a priori estimates (see~\cite{ColomboCrippaMarconiSpinolo3} for results in the case of the exponential kernel). The following theorem collects our main convergence results. 
\begin{theorem}\label{c:conv}
Assume that $u_0$, $V$ and $\eta$ satisfy~\eqref{e:u0}~\eqref{e:V} and~\eqref{e:eta}, respectively, and let  $w_\ee$  be as in~\eqref{e:w}, where $u_\ee$ is the solution of the Cauchy problem~\eqref{e:nlc}. Let $u$ be the entropy admissible solution of~\eqref{e:cl}. If \eqref{eqn:convex} holds and $\mathrm{TotVar} \ u_0 < + \infty$,
then~\eqref{e:vali} holds true. 

If furthermore $\eta (\xi) \xi \in L^1(\R)$
then we have the convergence rate 
\be \label{e:rate}
       \|  u(t, \cdot)  - w_\ee (t, \cdot) \|_{L^1(\R) } \leq C(\eta, V) \big[ \ee + \sqrt{\ee t} \big] \mathrm{TotVar} \ u_0 \quad  \text{for every $\ee>0$ and a.e. $t>0$} 
\eq
for a suitable constant $ C(\eta, V)$ only depending on $\eta$ and $V$. 
\end{theorem}
The convergence result in the above statement is a consequence of Theorem~\ref{t:tvb} and Theorem~\ref{t:entropy}, whereas to establish the convergence rate~\eqref{e:rate} we provide a fairly precise estimate on the entropy dissipation rate of nonlocal solutions (see Proposition~\ref{p:quasisolutions}) and then rely on an argument due to Kuznetsov~\cite{Kuznetsov}. 

To conclude we are left to discuss the role of the convexity assumption \eqref{eqn:convex}, which is the sole hypothesis in our convergence results Theorem~\ref{t:tvb} and Theorem~\ref{c:conv} that is not entirely standard in traffic models. It turns out that, if the convexity assumption \eqref{eqn:convex} is violated then the key estimate~\eqref{e:tv} fails in general. 
\begin{theorem}\label{t:cex}
Assume $V(w) : = 1-w$ and $\eta  : = \mathbbm{1}_{]-1, 0[}$; then for every sequence $\{ \ee_n \}$ satisfying 
\begin{equation} \label{e:ticino}
             \ee_n >0, \quad \ee_{n+1} \leq  \frac{1}{16} \ee_n \quad \text{ for every $n \in \mathbb N$}
\end{equation}
the following holds. There are an initial datum $u_0$ satisfying~\eqref{e:u0} and such that $\mathrm{TotVar} \ u_0~<~+ \infty$ and a sequence 
$\{ t_n \}$ such that for every $n \in \mathbb N$
\begin{equation} \label{e:gennaio}
          t_n >0, \qquad \mathrm{TotVar} \, w_{\ee_n} (t, \cdot) >   \mathrm{TotVar} \, w_{\ee_n} (0, \cdot)
         \quad \text{for a.e. $t \in \ ]0,  t_n[ $}. 
\end{equation}   
\end{theorem}
Some remarks are here in order. First, the function $V(w) =1-w$ satisfies~\eqref{e:V} and the function~$\eta~=~\mathbbm{1}_{]-1, 0[}$ satisfies~\eqref{e:eta}, but does not satisfy the convexity condition~\eqref{eqn:convex}.  Second, the counter-example given in Theorem~\ref{t:cex} is in contrast with numerical evidence provided in~\cite[\S5]{CocliteCoronDNKeimerPflug} and~\cite[p.~1949]{KeimerPflug2} which suggested that, when  $V(w) : = 1-w$ and $\eta  : = \mathbbm{1}_{]-1, 0[}$, $\mathrm{TotVar} \, w_{\ee_n} (t, \cdot)$ is a monotone non-increasing function of time.
{The numerical elusiveness is mainly due to the need of a specific choice of initial datum to enforce \eqref{e:gennaio} rather than to the numerical viscosity issue discussed in~\cite{ColomboCrippaGraffSpinolo}}. 
Our choice of initial datum is through an explicit formula, see~\eqref{e:idcex}, which depends on the chosen sequence $\{ \ee_n \}$. Third, we are reasonably confident that the basic ideas of our counter-example are not restricted to the case $\eta   = \mathbbm{1}_{]-1, 0[}$ and could be extended to a rather general class of kernels violating the convexity condition~\eqref{eqn:convex} but, {for simplicity}, we do not pursue this direction here. Finally, we refer to~\S\ref{ss:out} for an heuristic presentation of the counter-example. 
\subsection*{Paper outline}
In \S\ref{s:p} we overview some known preliminary results, in \S\ref{s:ptv}, \S\ref{s:en},  \S\ref{s:pc} and  \S\ref{s:cex} we establish Theorems~\ref{t:tvb},~\ref{t:entropy},~\ref{c:conv} and~\ref{t:cex}, respectively. For the reader's convenience we conclude the introduction by collecting the main notation used in paper. 
\subsection*{Notation}
\begin{itemize}
\item $\mathbbm{1}_E:$ the characteristic function of the measurable set $E$, i.e. $\mathbbm{1}_E(x) 
=1$ if $x \in E$, $\mathbbm{1}_E(x) 
=0$ if $x \notin E$;
\item $u\ast \eta$: the convolution of $u$ and $\eta$, computed with respect to the $x$ variable only, in other words  
$[u \ast \eta] (t, x)~=~\int_\R \eta (x-y) u (t, y) dy$;
\item $\mathrm{TotVar} \, u_0$: the total variation of the function $u_0$, see~\cite[\S3.2]{AmbrosioFuscoPallara}; 
\item $\R_-$, $\R_+$: the negative and the positive real axes, respectively, i.e. $\R_-=]-\infty, 0]$, $\R_+ = [0, + \infty[$;
\item $L^1_{\mathrm{loc}}(\R_+ \times \R)$: the space of functions $u$ such that $u \in L^1 (\Omega)$ for every open bounded set $\Omega \subseteq \R_+ \times \R$; 
\item $\mathrm{Lip}$: the space of Lipschitz continuous functions;
\item $C^{1, 1}$: the space of continuously differentiable functions with Lipschitz continuous derivatives;
\item $o(t)$ as $t \to 0$: the Bachmann-Landau notation for any function $g$ such that $\lim_{t \to 0} g(t)/t =0$.
\end{itemize}
\section{Preliminary results} \label{s:p}
Existence and uniqueness results for nonlocal conservation laws with anisotropic convolution kernels have been obtained in several works under different assumptions, see for instance \cite{BlandinGoatin,BressanShen,ChiarelloGoatin,CocliteDNKeimerPflug}. We now very slightly extend
\cite[Corollary 2.1]{CocliteDNKeimerPflug}
\begin{proposition} \label{p:exuni}
Assume that $u_0$, $V$ and $\eta$ satisfy~\eqref{e:u0},~\eqref{e:V} and~\eqref{e:eta}, respectively; then
\begin{itemize}
\item[i)] there is a unique distributional solution $u_\ee \in L^\infty(\R_+ \times \R)$ of the Cauchy problem~\eqref{e:nlc}. Also, 
\be \label{e:maxpu}
      0 \leq u_\ee \leq 1.
\eq
\item[ii)] $u_\ee \in C^0(\R_+, L^1_{\mathrm{loc}}(\R))$, namely the function $u_\ee$ admits a representative such that the map $t \mapsto u_\ee (t, \cdot)$ is continuous from $\R_+$ to $L^1_{\mathrm{loc}} (\R)$ endowed with the strong topology.
\item[iii)] If $u_0 \in \mathrm{Lip}(\R)$, then $u_\ee \in \mathrm{Lip}(\R_+ \times \R)$.
\end{itemize}
\end{proposition}
\begin{proof}
Concerning items i) and ii), the only difference with respect to~\cite[Corollary 2.1]{CocliteDNKeimerPflug} is that in that paper uniqueness is established in the more restrictive class $u_\ee \in L^\infty(\R_+ \times \R) \cap C^0(\R_+, L^1_{\mathrm{loc}}(\R))$. However, one can use the same argument as in the proof of ~\cite[Corollary 3.14]{DL:note} and~\cite[Proposition 2.3]{ColomboCrippaSpinolo} and show that, owing to renormalization, any bounded distributional solution of~\eqref{e:nlc} satisfies ii). Item iii) is a straightforward consequence of the analysis in~\cite{CocliteDNKeimerPflug} and of classical results concerning the method of characteristics. 
\end{proof}
As an easy consequence of Proposition~\ref{p:exuni} we get 
\begin{corol}
\label{c:w}
Assume that $u_0$, $V$ and $\eta$ satisfy~\eqref{e:u0},~\eqref{e:V} and~\eqref{e:eta}, respectively, and that $w_\ee$ is given by~\eqref{e:w}, where $u_\ee$ is the 
solution of the Cauchy problem~\eqref{e:nlc}. Then 
\begin{itemize}
\item[i)] $w_\ee \in \mathrm{Lip} (\R_+ \times \R)$ and 
\be \label{e:maxpw}
     0 \leq w_\ee \leq 1;
\eq
\item[ii)] $w_\ee \in C^0(\R_+, L^1_{\mathrm{loc}}(\R))$;
\item[iii)] if $u_0 \in \mathrm{Lip}(\R)$, then $w_\ee \in C^{1,1}(\R_+ \times \R)$.
\end{itemize}
\end{corol}
\begin{proof}
By combining~\eqref{e:eta},~\eqref{e:w} and~\eqref{e:maxpu} we get the maximum principle~\eqref{e:maxpw}. By differentiating~\eqref{e:w} we get 
\be
\label{e:wx}
     \frac{\partial w_\ee}{\partial x} = \frac{1}{\ee} \left[ \frac{1}{\ee} \int_x^{+ \infty} \eta'\left( \frac{x-y}{\ee} \right) u_\ee (t, y) dy - \eta(0^-) u_\ee (t, x)  \right]
\end{equation}
and 
\begin{equation}
\label{e:wt}
\begin{split}
           \frac{\partial w_\ee}{\partial t} & = -  \frac{1}{\ee} \int_{x}^{+\infty}  \eta \left( \frac{x-y}{\ee} \right) \partial_y [ u_\ee V(w_\ee) ]
            (t, y) dy\\
            & =    \frac{1}{\ee}\left[ \eta(0^-) u_\ee V(w_\ee)(t, x) - \frac{1}{\ee} \int_{x}^{+\infty}  \eta' \left( \frac{x-y}{\ee} \right)  u_\ee V(w_\ee)
          (t, y) dy \right].
\end{split}
\end{equation}
In the previous expression, $\eta(0^-)$ denotes the left limit of $\eta$ at $0$, which is well-defined since the restriction of $\eta$ to $\R_-$ is a monotone function. By combining~\eqref{e:wx} and~\eqref{e:wt} and using~\eqref{e:maxpu} and~\eqref{e:maxpw} we get 
\be \label{e:maxz}
     |\partial_x w_\ee (t, x)| \leq \frac{1}{\ee} \eta(0^-) , \quad 
      |\partial_t w_\ee (t, x)| \leq \frac{1}{\ee} \eta(0^-) \max_{w \in [0, 1]} |V(w) |
\quad \text{for every $(t, x) \in \R_+ \times \R$},
\eq
that is $w_\ee \in \mathrm{Lip} (\R_+ \times \R)$. Item ii) in the statement of Corollary~\ref{c:w} follows from item ii) of Proposition~\ref{p:exuni}. To establish item iii) we recall item iii) in the statement of Proposition~\ref{p:exuni} and use again~\eqref{e:wx} and~\eqref{e:wt}. 
\end{proof}
{We now recall a couple of well-known facts concering functions of bounded total variation that we need in the following. Assume $v \in L^\infty (\R)$ satisfies $\mathrm{Tot Var}\ v <+ \infty$; then 
\be \label{e:tvchar}
      \int_\R | v(x -  \xi)  - v(x)|  dx \leq |\xi| \mathrm{TotVar} \ v \quad \text{for every $\xi \in \R$}.
\eq
Also, let $\rho \in L^1(\R)$ satisfy 
\be \label{e:crho}
     C_\rho: = \int_\R |\rho (\xi) \xi | d\xi <+ \infty, \qquad \int_\R \rho(\xi) d \xi = 1 
\eq
and set $\rho_h (\xi) : = h^{-1}\rho (h^{-1} \xi)$; then 
\be \label{e:tvl1}
      \|  v \ast \rho_h - v \|_{L^1(\R)} \leq h \ C_\rho \mathrm{Tot Var}\ v
     \quad \text{for every $h>0$}. 
\eq
To conclude, we point out that, if  $\eta$ satisfies~\eqref{e:eta} then 
\be \label{e:etader}
\begin{split}
     \int_\R | \eta'(\xi) \xi| d\xi & = \lim_{R \to + \infty}  
      \int_{-R}^0 | \eta'(\xi) \xi| d\xi \stackrel{\eta'\ge 0}{=}
   -    \lim_{R \to + \infty}  \int_{-R}^0 \eta'(\xi) \xi d\xi \\&
    \stackrel{\text{integration by parts}}{=}
 - \lim_{R \to + \infty} R \eta(-R) + \lim_{R \to + \infty} 
    \int_{-R}^0 \eta(\xi) d\xi \leq  \lim_{R \to + \infty} 
    \int_{-R}^0 \eta(\xi) d\xi \stackrel{\eqref{e:eta}}{\leq} 1.
\end{split}
\eq}
\section{Proof of Theorem~\ref{t:tvb}} \label{s:ptv}
By combining~\eqref{e:wx} and~\eqref{e:wt} we get that the material derivative of $w_\ee$ is given by 
\begin{equation}
\label{e:ewgen}
      \partial_t w_\ee + V(w_\ee) \partial_x w_\ee = 
      \frac{1}{\ee^2} \int_x^{+\infty} \eta' \left( \frac{x-y}{\ee} \right) [V(w_\ee (t, x)) - V(w_\ee (t, y) )] u_\ee (t, y)dy. 
\end{equation}
To avoid some technicalities, we first provide the proof of Theorem~\ref{t:tvb} under some further assumptions on the initial datum $u_0$ and the convolution kernel $\eta$. We then remove these assumptions by relying on a fairly standard approximation argument. \\
{\sc Step 1:} we impose the further assumptions $u_0 \in \mathrm{Lip}( \R)$, $\eta \in C^2 (]-\infty, 0[)$, $\eta'' \in L^1 (]-\infty, 0[)$. Owing to iii) in the statement of Corollary~\ref{c:w}, this implies that $w_\ee \in C^{1,1} (\R_+ \times \R)$. We set $z_\ee : = \partial_x w_\ee$ and point out that $z_\ee \in \mathrm{Lip} (\R_+ \times \R)$. By differentiating~\eqref{e:ewgen} we get 
\begin{equation}\label{e:z_ee}
\begin{split}
\partial_t z_\ee + \partial_x (V(w_\ee) z_\ee) = 
 \frac{1}{\ee^2}\int_x^{+\infty} \biggl[ & \frac{1}{\ee}\eta''\left( \frac{x-y}{\ee}\right)\big(V(w_\ee(t,x))- V(w_\ee(t,y))\big)  \\
&  + \eta'\left( \frac{x-y}{\ee}\right) V'(w_\ee(t,x)) z_\ee(t,x) \biggr] u_\ee(t,y) dy.
 \end{split}
\end{equation}
Note that the right-hand side of the above expression is finite since $w_\ee$ and $u_\ee$ are both bounded functions, and both $\eta''$ and $\eta'$ are summable. 
Assume for a moment we have shown that 
\be 
\label{e:finitetv}
      \mathrm{Tot Var} \, w_\ee (t, \cdot) =
       \int_\R |z_\ee (t, x)| < + \infty \quad \text{for every $t>0$ and $\ee>0$};
\eq
then by multiplying~\eqref{e:z_ee} by $s(t, x) := \mathrm{sign}(z_\ee(t, x))$ and $x$-integrating over $\R$ we arrive at 
\begin{equation*}
\begin{split}
\frac{d}{dt} \int_\R \! \! |z_\ee(t,x)| dx  \! + \!\int_\R \! \partial_x (V(w_\ee) |z_\ee|) dx& = 
 \frac{1}{\ee^2}\int_\R s(t,x) \int_x^{+\infty} \biggl[  \frac{1}{\ee}\eta''\left( \frac{x-y}{\ee}\right)\big(V(w_\ee(t,x))\!-\! V(w_\ee(t,y))\big)  \\
&  \qquad + \eta'\left( \frac{x-y}{\ee}\right) V'(w_\ee(t,x)) z_\ee(t,x) \biggr] u_\ee(t,y) dy  dx.
\end{split}
\end{equation*}
By relying on Fubini's theorem we rewrite the above equality as 
\begin{equation} \label{e:fubini}
\begin{split}
\frac{d}{dt} \mathrm{TotVar} \, w_\ee (t, \cdot) =
\frac{d}{dt}  \int_\R \! \! |z_\ee(t,x)| dx& = \frac{1}{\ee^2}\int_\R u_\ee(t,y) \sigma(t,y) dy,
\end{split}
\end{equation}
where
\begin{equation*}
\sigma(t,y) : =  \int_{-\infty}^y \! \! \frac{1}{\ee}\eta''\left( \frac{x-y}{\ee}\right)\big(V(w_\ee(t,x))- V(w_\ee(t,y))\big) s(t,x) dx 
 +\int_{-\infty}^{y} \! \! \eta'\left( \frac{x-y}{\ee}\right) V'(w_\ee(t,x)) |z_\ee(t,x)| dx.
\end{equation*}
By using the integration by parts formula we arrive at 
\be \label{e:sigma}
     \sigma(t,y) : =  \int_{-\infty}^y \! \! \frac{1}{\ee}\eta''\left( \frac{x-y}{\ee}\right)
      \Big[
     \big(V(w_\ee(t,x))- V(w_\ee(t,y))\big) s(t,x)  - \gamma(t, x, y)  \Big]dx, 
\eq
where
$$
     \gamma (t, x, y) : = - \int_x^y V'(w_\ee(t,\xi)) |z_\ee(t,\xi)| d \xi. 
$$
Note that $\gamma \ge 0$ since $x\leq y$ and $V' \leq 0$. Also, 
$$
    |V(w_\ee(t,x))- V(w_\ee(t,y))| \leq 
     \left| \int_y^x  V'(w_\ee(t,\xi)) z_\ee(t,\xi) d \xi   \right| \leq \gamma(t, x, y)
$$
and by recalling~\eqref{e:sigma} and the inequality $\eta''\ge0$ we conclude that $\sigma (t, y) \leq 0$. By plugging this inequality into~\eqref{e:fubini} and recalling that $u_\ee \ge 0$ owing to~\eqref{e:maxpu} we conclude that the right-hand side of~\eqref{e:fubini} is nonpositive, which in turn yields~\eqref{e:tv}. \\
{\sc Step 2:} we establish~\eqref{e:finitetv} under the assumptions that $u_0 \in \mathrm{Lip}(\R)$ and $\eta \in C^2 (]-\infty, 0[)$, $\eta'' \in L^1 (]-\infty, 0[)$. We recall that $z_\ee = \partial_x w_\ee$ and it is bounded by the first inequality in~\eqref{e:maxz}.
Next, we fix $R>0$, multiply~\eqref{e:z_ee} by $s(t, x) := \mathrm{sign}(z_\ee(t, x))$ and $x$-integrate over the interval $]-R, R[$. By performing a change of variables and applying Fubini's theorem we arrive at 
\begin{equation} \label{e:erre}
\begin{split}
     \frac{d}{dt} \int_{-R}^R |z_\ee (t, x)| dx & +
     \int_{-R}^R \partial_x (V(w_\ee)|z_\ee|) (t, x) dx 
     \leq \underbrace{
     \frac{1}{\ee^2}  \int_{\R_-}  \eta''(\xi) 
     \int_{-R}^{R}\big| V(w_\ee(t,x))\!-\! V(w_\ee(t,x - \ee \xi))\big| dx d \xi}_{:= T_1} \\
&   + \underbrace{\frac{1}{\ee}  \int_{\R_-} \eta' (\xi) \int_{-R}^{R} |V'(w_\ee(t,x))|| z_\ee(t,x) | u_\ee(t,x-\ee \xi) dx d\xi}_{:=T_2}.
\end{split}
\end{equation}
Note that 
\begin{equation} \label{e:t1}
\begin{split}
      T_1 & \leq  \underbrace{\frac{1}{\ee^2}  \int_{\R_-}  \eta''(\xi) 
     \int_{-R}^{R+ \ee \xi}\big| V(w_\ee(t,x))\!-\! V(w_\ee(t,x - \ee z))\big| dx d \xi}_{T_{11}} \\
     & + \underbrace{\frac{1}{\ee^2}  \int_{\R_-}  \eta''(\xi) 
     \int_{R+ \ee \xi }^{R}\big| V(w_\ee(t,x))\!-\! V(w_\ee(t,x - \ee z))\big| dx d \xi.}_{T_{12}}
\end{split}
\end{equation}
We have 
\be \label{e:t12}
    T_{12} \stackrel{\eqref{e:maxpw}}{\leq} -
     \frac{2}{\ee} \max_{w \in [0, 1]}|V(w)| \int_{\R_-}  \eta''(\xi) 
     \xi d\xi \stackrel{\text{integration by parts},~\eqref{e:etader}}{=} \frac{2}{\ee} \max_{w \in [0, 1]}|V(w)| \eta(0^-).
\eq
To control the term $T_{11}$, we point out that, if $\xi\leq 0$ and $x \in ]-R, R + \ee \xi[$, then both $x$ and $x-\ee  \xi$ belong to the interval $]-R, R[$. This implies that 
\be \label{e:t11}
   T_{11} \leq \frac{1}{\ee} \| V' \|_{L^\infty} \left( \int_{\R_-}  \eta''(\xi)  |\xi | d \xi \right) 
   \left( \int_{-R}^R |z_\ee (t, x)| dx \right)=
   \frac{1}{\ee} \| V' \|_{L^\infty} \eta (0^-)  
    \int_{-R}^R |z_\ee (t, x)| dx.
\eq
We also have 
\be \label{e:t2}
    T_2  \stackrel{\eqref{e:maxpu}}{\leq} \frac{1}{\ee} \eta(0^-) \essup_{w \in ]0, 1[}| V' (w)| \int_{-R}^R |z_\ee (t, x)| dx.
\eq
We plug~\eqref{e:t12} and \eqref{e:t11} into~\eqref{e:t1}, combine them with~\eqref{e:erre} and~\eqref{e:t2}, recall~\eqref{e:maxz} and apply Gronwall's Lemma. We obtain a bound on $\int_{-R}^R |z_\ee (t, x)| dx $ which does not depend on $R$ (albeit it depends on $\ee$). By sending $R\to + \infty$ we establish~\eqref{e:finitetv}. \\
{\sc Step 3:} we remove the assumptions $u_0 \in \mathrm{Lip}(\R)$ and $\eta \in C^2 (]-\infty, 0[)$, $\eta'' \in L^1 (]-\infty, 0[)$. We fix $u_0$ and $\eta$ as in the statement of Theorem~\ref{t:tvb} and term $u_\ee$ the solution of the Cauchy problem~\eqref{e:nlc} and $w_\ee$ the corresponding convolution terms. Next, we 
construct suitable sequences $\{ u_{0n} \}$ and $\{ \eta_n\}$ satisfying the assumptions of Theorem~\ref{t:tvb}, the further conditions $u_{0n} \in \mathrm{Lip} (\R)$, $\eta_n \in C^2 (]-\infty, 0[)$, $\eta''_n \in L^1 (]-\infty, 0[)$ and such that 
\be \label{e:conv}
       u_{0n} \weaks u_0 \; \text{weakly$^\ast$ in $L^\infty (\R)$}, \quad \mathrm{Tot Var} \, u_{0n} \leq \mathrm{Tot Var} \, u_{0}, \quad 
      \lim_{n \to + \infty}  \| u_{n0} - u_n \|_{L^1 (\R)} =0
\eq
and that 
\be \label{e:conv2}
        \eta_n \to \eta  \; \text{strongly in $L^1 (\R)$}, \quad \{ \eta_n(0) \} \; \text{is uniformly bounded}.
\eq
We can for instance define $u_{0n}$  by convolving $u_0$  with a suitable family of convolution kernels $\{ \rho_n \}$ and then use~\eqref{e:tvl1} to obtain the last condition in~\eqref{e:conv}. 

We term $u_{\ee n}$ the sequence of solutions of the Cauchy problem~\eqref{e:nlc} and $w_{\ee n}$ the corresponding sequence of convolution terms. By recalling~\eqref{e:maxpu} and by combining~\eqref{e:maxz} with~\eqref{e:conv2} and the Ascoli-Arzel\`a theorem we get that 
$$
    u_{\ee n} \weaks v_\ee \quad \text{weakly$^\ast$ in $L^\infty (\R_-\times \R)$}, 
     \qquad w_{\ee n} \to p_\ee \quad \text{in $C^0(K)$, for every $K \subseteq \R_-\times \R$ compact}
$$
for suitable functions $v_\ee \in L^\infty (\R_-\times \R)$ and $p_\ee \in \mathrm{Lip} (\R_-\times \R)$.  Owing to~\eqref{e:conv2} we can pass to the limit in the equality $w_{\ee n} (t, x)= \int_{\R_-}  \eta_n(\xi) u_{\ee n} (t, x-\ee \xi) d \xi$  and arrive at $p_{\ee } (t, x)= \int_{\R_-}  \eta(\xi) v_\ee (t, x-\ee \xi) d \xi$. Also, by passing to the limit in the distributional formulation we infer that $v_\ee$ is a bounded distributional solution of the Cauchy problem~\eqref{e:nlc} and by uniqueness this implies $v_\ee =u_\ee$, $p_\ee = w_\ee$. By applying {\sc Step 1} to the sequence $w_{\ee n}$, recalling the lower semicontinuity of the total variation and then passing to the limit we arrive at 
\be \label{e:quasi}
    \mathrm{Tot Var} \, w_\ee (t, \cdot) \leq \liminf_{n \to + \infty}   \mathrm{Tot Var} \, w_{\ee n} (0, \cdot).
\eq
By combining~\eqref{e:conv} with~\eqref{e:conv2} and~\eqref{e:wx} we get 
$$
    \lim_{n \to + \infty} \mathrm{Tot Var} \, w_{\ee n} (0, \cdot)=  \mathrm{Tot Var} \, w_{\ee } (0, \cdot)
$$
and by plugging the above equality into~\eqref{e:quasi} we get~\eqref{e:tv}.
\section{Proof of Theorem~\ref{t:entropy}} \label{s:en}
We need the following lemma. 
{\begin{lemma}\label{l:compactnessL1}
Let $\eta$ satisfy~\eqref{e:eta} and assume that $\{ v_k \}$ is a pre-compact sequence in $L^1_{\mathrm{loc}}(\R_+ \times \R)$ such that $\| v_k \|_{L^\infty} \leq \Lambda$ for some $\Lambda>0$ and for every $k$. Set 
$$
    F_{k \ee}(t,x) := \int_x^{+\infty} \frac1\ee \eta\left( \frac{x-y}\ee\right)|v_k(t,y) - v_k(t,x)| dy;
$$
then the family $\{ F_{k \ee} \}$ converges to $0$ as $\ee \to 0^+$ in $L^1_{\mathrm{loc}}(\R_+ \times \R)$, uniformly in $k$. In other words, for every  $r, T, M>0$, there is $\tilde \ee>0$, depending on $r, T$ and $M$ only, such that if $\ee \leq \tilde \ee$, then 
\[
\int_0^T \int_{-M}^M | F_{k \ee}(t,x)| dx dt \leq r \quad \text{for every $k$}. 
\]
\end{lemma}
\begin{proof}
We proceed according to the following steps. \\
{\sc Step 1:} we rely on the Fr\'echet-Kolmogorov theorem and infer that the sequence $\{ v_k \}$ is equicontinuous in $L^1_{\mathrm{loc}}$. In particular, for every $T>0$, $L>0$ and $\nu>0$ there is $\tilde \tau (T, L, \nu)>0$ such that, if $|\tau|< \tilde \tau (T, L, \nu)$, then 
\be \label{e:equic2}
     \int_0^T \int_{-L +  \tau}^{L-  \tau} |v_k (t, x+ \tau) - v_k (t, x) | dx dt \leq \nu
     \quad \text{for every $k$.}
\eq
{\sc Step 2:} we now fix $\delta>0$, choose $R>0$ in such a way that $\int_{-\infty}^{-R} \eta(z) dz < \delta$ and point out that 
\be \label{e:equic}
\begin{split}
\int_0^T \int_{-M}^M & | F_{k\ee}(t,x)| dx dt 
\stackrel{\xi = \frac{x-y}{\ee}}{=}  \int_0^T \int_{-M}^M 
    \int_{\R_-} \eta(\xi)|v_k(t,x- \ee \xi) - v_k(t,x)| d\xi dx dt \\
&\stackrel{\text{Fubini}}{=}\int_{-R}^0 \! \!   \eta(\xi) \int_0^T \! \! \int_{-M}^M \! \!
    |v_k(t,x- \ee \xi) - v_k(t,x)|  dx dt d\xi \\ & 
    \qquad + \int_{-\infty}^{-R}\! \!  \eta(\xi) \int_0^T\! \! \int_{-M}^M \! \!
   |v_k(t,x- \ee \xi) - v_k(t,x)| dx dt d\xi \stackrel{\eqref{e:equic2}}{\leq} \nu + 4 \delta TM \underbrace{\| v_k \|_{L^\infty}}_{\leq \Lambda}  
\end{split}
\eq
provided $L=M+1$ and $\ee R < \min\{\tilde \tau (T, L, \nu),  1 \}$. By the arbitrariness of $\nu$ and $\delta$ this concludes the proof of Lemma~\ref{l:compactnessL1}. 
\end{proof}}
We are now ready to provide the actual proof of Theorem~\ref{t:entropy}.
\begin{proof}[Proof of Theorem~\ref{t:entropy}]
We fix an entropy-entropy flux pair $(\alpha, \beta)$ for~\eqref{e:cl}, namely $\alpha, \beta: \R \to \R$ are two $C^2$ functions satisfying $\alpha '' \ge 0$, $\beta' (u)=\alpha'(u) [ V(u) + u V'(u)]$. Also, we fix $\phi \in C^\infty_c (\R_+ \times \R)$ satisfying $\phi\ge 0$. By the arbitrariness of $(\alpha, \beta)$ and $\phi$, the proof of Theorem~\ref{t:entropy} boils down to the proof of the inequality 
\be \label{e:ineq}
D_\alpha(\phi) : =\iint_{\R_+ \times \R}  \big[ \alpha(u) \partial_t \phi+ \beta(u) \partial_x \phi \big] dx dt + \int_\R \alpha(u_0(x))\phi(0,x) dx \ge 0. 
\eq
Note that in particular~\eqref{e:ineq} implies that $u$ is a distributional solution of~\eqref{e:cl}. Indeed,  by choosing $(\alpha(u), \beta(u))= (u, uV(u))$ and $(\alpha(u), \beta(u))= (-u, -uV(u))$ we get the inequalities $\partial_t u + \partial_x [uV] \leq 0$ and 
$\partial_t u + \partial_x [uV] \ge 0$, respectively, which are satisfied in the sense of distributions, and from this we infer~\eqref{e:cl}. 

To establish~\eqref{e:ineq} we set $\eta_\ee (\xi) : = \frac{1}{\ee} \eta\left( \frac{\xi}{\ee}\right)$, choose a sequence $\{\ee_k \}$ and set 
\be \label{e:dalpha}
      D^{\ee_k}_\alpha(\phi) : =\iint_{\R_+ \times \R}   \big[ \alpha(w_{\ee_k}) \partial_t \phi+ \beta(w_{\ee_k}) \partial_x \phi \big] dx dt +
       \int_\R \alpha(w_{\ee_k}(0, x)) \phi(0, x)dx . 
\eq
We recall that by assumption $w_{\ee_k}\to w$ in $L^1(]0,T[\times \R)$ and that $w_{\ee_k}(0,x)= u_0\ast \eta_{\ee_k}\to u_0$ in $L^1(\R)$. This implies that $\lim_{k \to + \infty}  D^{\ee_k}_\alpha(\phi) = D_\alpha(\phi)$ and hence that establishing~\eqref{e:ineq} amounts to show that 
\be \label{e:ineq2}
     \lim_{k \to + \infty}  D^{\ee_k}_\alpha(\phi) \ge 0. 
\eq 
To establish~\eqref{e:ineq2} we proceed according to the following steps. \\
{\sc Step 1:} we write $D^{\ee_k}_\alpha(\phi)$ in a more convenient form. To this end we 
consider the equation at the first line of~\eqref{e:nlc}, convolve it with $\eta_{\ee_k}$ and then multiply the result times $\alpha' (w_{\ee_k})$. We arrive at 
\[
\partial_t \alpha(w_{\ee_k}) + \alpha'(w_{\ee_k}) \partial_x \big[[ V(w_{\ee_k}) u_{\ee_k} ]\ast \eta_{\ee_k}\big] = 0. 
\]
Next, we multiply the above equation times $\phi$ and integrate over $\R_+ \times \R$. Owing to the integration by parts formula this yields
\begin{equation}\label{e:alphaweek}
 \iint _{\R_+ \times \R} \alpha(w_{\ee_k})  \partial_t \phi + [(V(w_{\ee_k})u_{\ee_k})\ast\eta_{\ee_k}] \partial_x [ \alpha'(w_{\ee_k})\phi ]dx dt + \int_\R \alpha (w_{\ee_k}(0,x)) \phi(0,x) dx = 0.
\end{equation}
We now point that, owing to the equality $\beta' (u)=\alpha'(u) [ V(u) + u V'(u)]$, we have 
\be \label{e:pino}
\begin{split}
         \iint _{\R_+ \times \R}  \beta(w_{\ee_k}) \partial_x \phi \ dx dt & = -   \iint _{\R_+ \times \R} \beta' (w_{\ee_k}) \partial_x w_{\ee_k} \phi \ dx dt
        \\ & = -  \iint _{\R_+ \times \R} \alpha'(w_{\ee_k}) [ V(w_{\ee_k}) + w_{\ee_k} V'(w_{\ee_k})] \partial_x w_{\ee_k} \phi \ dx dt \\ & = 
        -   \iint _{\R_+ \times \R} \alpha'(w_{\ee_k}) \partial_x [ V(w_{\ee_k})  w_{\ee_k}]  \phi \ dx dt =
       \iint _{\R_+ \times \R}   V(w_{\ee_k})  w_{\ee_k} \partial_x [  \alpha'(w_{\ee_k}) \phi]  dx dt .
\end{split}
\eq
By comparing~\eqref{e:dalpha} with~\eqref{e:alphaweek} and using~\eqref{e:pino} we then get 
\begin{equation}\label{e:mualpha}
\begin{split}
D^{\ee_k}_\alpha(\phi) &
\stackrel{\eqref{e:alphaweek},\eqref{e:pino}}{=}
   \iint _{\R_+ \times \R} \underbrace{ \big[  V(w_{\ee_k})w_{\ee_k} - [ V(w_{\ee_k})u_{\ee_k}] \ast \eta_{\ee_k}  \big] }_{S_\ee (t, x)}
  \partial_x[ \alpha'(w_{\ee_k})\phi] dx dt \\
  &=   \iint _{\R_+ \times \R}  S_\ee (t, x)  \alpha'(w_{\ee_k}) \partial_x \phi dx dt + 
  \underbrace{  \iint _{\R_+ \times \R}  S_\ee (t, x)  \partial_x [\alpha'(w_{\ee_k}) ] \phi dx dt}_{T_1^{\ee_k}} . 
\end{split}
\end{equation}
Note that \begin{equation}\label{e:mualpha2}
\begin{split}
 T_1^{\ee_k} & =
  \iint _{\R_+ \times \R} \int_x^{+ \infty} \partial_x[ \alpha'(w_{\ee_k})](t,x) \phi (t,x)\frac1{\ee_k}\eta\left(\frac{x-y}{{\ee_k}}\right) u_{\ee_k}(t,y) \Big[ V(w_{\ee_k}(t,x)) -  V(w_{\ee_k}(t,y)) \Big] dy dx dt \\
=&  \iint _{\R_+ \times \R} u_{\ee_k}(t,y) \omega_{\ee_k} (t,y) dy dt,
\end{split}
\end{equation}
provided 
\[
\begin{split}
\omega_{\ee_k}(t,y) := &~  \int_{-\infty}^y \partial_x [\alpha'(w_{\ee_k})](t,x) \phi (t,x)\frac1{\ee_k}\eta\left(\frac{x-y}{{\ee_k}}\right) 
   \Big[ V(w_{\ee_k}(t,x)) -  V(w_{\ee_k}(t,y))  \Big] dx  \\
= &~ \int_{-\infty}^y \underbrace{\partial_x [ \alpha'(w_{\ee_k})](t,x)V(w_{\ee_k}(t,x))}_{\partial_x I(w_{\ee_k})}  \phi (t,x)\frac1{\ee_k}\eta\left(\frac{x-y}{{\ee_k}}\right) dx \\
&~ - V(w_{\ee_k}(t,y)) \int_{-\infty}^y \partial_x[ \alpha'(w_{\ee_k})](t,x) \phi (t,x)\frac1{\ee_k}\eta\left(\frac{x-y}{{\ee_k}}\right) dx.
\end{split}
\]
In the previous expression, we have chosen the $C^1$ function $I$ in such a way that $I'(u) = \alpha''(u) V(u)$. By applying the integration by parts formula we then get 
\begin{equation} \label{e:omega}
\begin{split}
        \omega_{\ee_k}(t,y) = &I (w_{\ee_k}(t,y))\phi(t,y) \frac{\eta(0)}{\ee_k} - \int_{-\infty}^y  I(w_{\ee_k}(t,x)) \partial_x \left[  \phi (t,x)\frac1{\ee_k}\eta\left(\frac{x-y}{{\ee_k}}\right)\right] dx \\
& -  V(w_{\ee_k}(t,y)) \left[ \alpha'(w_{\ee_k}(t,y))\phi(t,y) \frac{\eta(0)}{\ee_k} - 
     \int_{-\infty}^y \alpha'(w_{\ee_k}(t,x))  \partial_x \left[  \phi (t,x)\frac1{\ee_k}\eta\left(\frac{x-y}{{\ee_k}}\right)\right] dx \right] \\
= & ~\int_{-\infty}^y  \big[ I(w_{\ee_k}(t,y)) - I (w_{\ee_k}(t,x)) \big]  \partial_x \left[  \phi (t,x)\frac1{\ee_k}\eta\left(\frac{x-y}{{\ee_k}}\right)\right] dx\\
&~  -  V(w_{\ee_k}(t,y)) \int_{-\infty}^y   \big[ \alpha'(w_{\ee_k}(t,y)) -  \alpha'(w_{\ee_k}(t,x))\big]\partial_x \left[  \phi (t,x)\frac1{\ee_k}\eta\left(\frac{x-y}{{\ee_k}}\right)\right] dx \\& 
= G^1_{\ee_k} (t, y) +G^2_{\ee_k} (t, y) + P_{\ee_k} (t, y), \phantom{\int}
\end{split}
\end{equation}
provided 
\begin{eqnarray} \label{e:G}
   G^1_{\ee_k} (t, y) : =
     \int_{-\infty}^y  \big[ I(w_{\ee_k}(t,y)) - I (w_{\ee_k}(t,x)) \big]  \frac1{\ee_k}\eta\left(\frac{x-y}{{\ee_k}}\right) \partial_x  \phi (t,x) dx \\ \nonumber
   G^2_{\ee_k} (t, y) =  - V(w_{\ee_k}(t,y)) \int_{-\infty}^y   \big[ \alpha'(w_{\ee_k}(t,y)) -  \alpha'(w_{\ee_k}(t,x))\big]\frac1{\ee_k}\eta\left(\frac{x-y}{{\ee_k}}\right)
   \partial_x \phi (t,x)dx
\end{eqnarray}
and 
\begin{equation} \label{e:pi}\begin{split}
 P_{\ee_k} (t, y) : & = \int_{-\infty}^y H(w_\ee(t,x),w_\ee(t,y)) \phi(t,x) \partial_x \left[ \frac1\ee\eta\left(\frac{x-y}{\ee}\right)\right]  dx\\
 & = \frac{1}{\ee^2}
\int_{-\infty}^y H(w_\ee(t,x),w_\ee(t,y)) \phi(t,x) \eta'\left(\frac{x-y}{\ee}\right)  dx ,
\end{split}
\end{equation}
where
\be \label{e:acca}
H(a,b) : = I(b) - I(a) - V(b)(\alpha'(b) - \alpha'(a)).
\eq
By plugging~\eqref{e:G} and~\eqref{e:pi} into~\eqref{e:omega} and then recalling~\eqref{e:mualpha} and~\eqref{e:mualpha2} we then arrive at 
\be
\label{e:perdopo}
    D^{\ee_k}_\alpha(\phi) =
      \iint _{\R_+ \times \R}  S_\ee (t, x)  \alpha'(w_{\ee_k}) \partial_x \phi dx dt + 
       \iint _{\R_+ \times \R} u_{\ee_k}(t,y) [ G^1_{\ee_k} (t, y) +G^2_{\ee_k} (t, y) +
      P_{\ee_k} (t, y) ] dy dt
\eq
{\sc Step 2:} we establish~\eqref{e:ineq2}. \\
{\sc Step 2A:} we show that 
\be \label{e:perdopo2}
      D^{\ee_k}_\alpha(\phi) \ge 
     \underbrace{\iint_{\R_+ \times \R}  S_\ee (t, x)  \alpha'(w_{\ee_k}) \partial_x \phi dx dt}_{T^{\ee_k}_2} + 
     \underbrace{ \iint_{\R_+ \times \R}  u_{\ee_k}(t,y) G^1_{\ee_k} (t, y)dy dt}_{T^{\ee_k}_3} +
      \underbrace{ \iint_{\R_+ \times \R}  u_{\ee_k}(t,y) G^2_{\ee_k} (t, y)] dy dt}_{T^{\ee_k}_4} .
\eq
Owing to~\eqref{e:perdopo} and since $u_{\ee_k} \ge 0$ it suffices to show that $P_{\ee_k} \ge 0$. To this end we recall~\eqref{e:pi} and~\eqref{e:acca} 
and that $I'(u) = \alpha''(u) V(u)$. We point out that 
$$
     \frac{\partial H}{\partial a} (u, b) = - I'(u) + V(b) \alpha''(u) = \alpha''(u) [V(b) - V(u)] 
$$
and by recalling that $\alpha''\ge 0$ and that $V' \leq 0$ we conclude that the last expression is non-positive if $u\leq b$ and non-negative if $u \ge b$. This implies that $u=b$ is a minimum for $H(u, b)$ and since $H(b, b)=0$ we obtain the inequality $H(a, b) \ge0$ for every $(a, b) \in \R^2.$ By plugging this information into~\eqref{e:pi} and recalling that $\phi \ge 0$ and that $\eta'\ge 0$ by~\eqref{e:eta}, we conclude that $P_{\ee_k} \ge 0$.   \\
{\sc Step 2B:} to establish~\eqref{e:ineq2} it suffices to show that the right-hand side of~\eqref{e:perdopo2} vanishes in the $k \to + \infty$ limit. To this end we recall the explicit expression~\eqref{e:mualpha} of $S_{\ee_k}$ and point out that 
\be \label{e:esse}
\begin{split}
S_{\ee_k} (t,x) =  &~ \int_x^{+\infty}\frac1{\ee_k}\eta\left(\frac{x-y}{\ee_k}\right) \Big[
    w_{\ee_k} (t,x)V(w_{\ee_k}(t,x)) - u_{\ee_k}(t,y) V(w_{\ee_k}(t,y)) \Big] dy \\
= &~ \int_x^{+\infty}\frac1{\ee_k}\eta\left(\frac{x-y}{\ee_k}\right) u_{\ee_k}(t,y) \Big[  V(w_{\ee_k}(t,x)) -  V(w_{\ee_k}(t,y)) \Big] dy. 
\end{split}
\eq
Owing to~\eqref{e:maxpu}, $\|u_{\ee_k}(t,y) \|_{L^\infty} \leq 1$; by applying Lemma~\ref{l:compactnessL1} with $v_k: = V(w_{\ee_k})$ we get that $S_{\ee_k}$ converges to $0$ in {$L^1_{\mathrm{loc}} (\R_+ \times \R)$. Since $\phi$ is compactly supported, this implies that $\lim_{k \to + \infty} T^{\ee_k}_2=0$.
We now show that $\lim_{k \to + \infty} T^{\ee_k}_3=0$. To this end, we point out that 
\[
\begin{split}
|T^{\ee_k}_3 |& \stackrel{\eqref{e:G},\eqref{e:perdopo2}}{\leq}
\iint_{\R_+ \times \R}  u_{\ee_k}(t,y) \int_{-\infty}^y  \big| I(w_{\ee_k}(t,y)) - I (w_{\ee_k}(t,x)) \big|  \frac1{\ee_k}\eta\left(\frac{x-y}{{\ee_k}}\right) | \partial_x  \phi (t,x)| dx dy  dt \\
& \stackrel{\text{Fubini}}{=}   \iint_{\R_+ \times \R}  | \partial_x  \phi (t,x)| \int_x^{+\infty}  \underbrace{u_{\ee_k}(t,y)}_{\leq 1 \; \text{by~\eqref{e:maxpu}}}  \big| I(w_{\ee_k}(t,y)) - I (w_{\ee_k}(t,x)) \big|  \frac1{\ee_k}\eta\left(\frac{x-y}{{\ee_k}}\right) dy dx dt.
\end{split}
\]
Since $\phi$ is compactly supported, by applying Lemma~\ref{l:compactnessL1} with $v_k: = I(w_{\ee_k})$ we conclude that $T^{\ee_k}_3 $ vanishes in the $\ee \to 0^+$ limit.  By relying on an analogous argument we show that $\lim_{k \to + \infty} T^{\ee_k}_4=0$ and owing to~\eqref{e:perdopo2} this concludes the proof of~\eqref{e:ineq2}. }
\end{proof}
\section{Proof of Theorem~\ref{c:conv}} \label{s:pc}
\subsection{Convergence proof} \label{ss:cp}
Let $u$ be the entropy admissible solution of~\eqref{e:cl}, we now establish~\eqref{e:vali}. \\
{{\sc Step 1:} we show that the family $ \{ w_\ee \}$ is pre-compact in the strong topology of $L^1_{\mathrm{loc}} (\R_+ \times \R)$.   To this end, we point out that, for every $T>0$, we have 
\be \label{e:wxtx}
    \int_0^T \int_\R |\partial_x w_\ee| (t, x) dxdt \stackrel{\eqref{e:tv}}{\leq} 
      T \mathrm{TotVar} w_\ee (0, \cdot) \stackrel{\eqref{e:tv0}}{\leq} T \, \mathrm{TotVar} \ u_0 
\eq
and 
\begin{equation*}
\begin{split}
         \int_0^T \int_\R& \left| \frac{1}{\ee^2} \int_x^{+\infty} 
          \eta' \left( \frac{x-y}{\ee} \right) [V(w_\ee (t, x)) - V(w_\ee (t, y) )] u_\ee (t, y)dy
          \right| dx dt \\ &
             \stackrel{\xi = \frac{x-y}{\ee}}{=}
         \frac{1}{\ee} \int_0^T \int_\R \left|  \int_{\R_-} \eta'(\xi)  
         [V(w_\ee (t, x -\ee \xi)) - V(w_\ee (t, y) )] u_\ee (t, x -\ee \xi)d \xi 
          \right| dx dt \\ &
\leq  \frac{1}{\ee}\int_0^T \int_{\R^-} | \eta'(\xi) | \int_{\R}  
         |V(w_\ee (t, x -\ee \xi)) - V(w_\ee (t, y) )| | u_\ee (t, x -\ee \xi)|d x  d\xi dt \\ &
        \stackrel{\eqref{e:maxpu},\eqref{e:tvchar}}{\leq} 
         T  \sup_{t \in ]0, T[} \mathrm{TotVar} [ V\circ w_\ee (t, \cdot) ] \int_{\R^-} |\eta'(\xi) \xi| d \xi \\& 
        \stackrel{\eqref{e:maxpw}}{\leq}
           T \, \essup_{w \in ]0, 1[} |V'(w)| \sup_{t \in ]0, T[}\mathrm{TotVar} \,  w_\ee (t, \cdot)  \int_{\R^-} |\eta'(\xi) \xi| d \xi \\
         &
         \stackrel{\eqref{e:tv},\eqref{e:tv0}}{\leq}
          T \, \essup_{w \in ]0, 1[} |V'(w)| \mathrm{TotVar} \,  u_0   \int_{\R^-} |\eta'(\xi) \xi| d \xi
          \stackrel{\eqref{e:etader}}{\leq}    T \, \sup_{w \in ]0, 1[} |V'(w)| \mathrm{TotVar} \,  u_0. 
\end{split}
\end{equation*}
Owing to~\eqref{e:ewgen} and recalling~\eqref{e:wxtx}, this implies that $\int_0^T \int_\R |\partial_t w_\ee| dx dt$ is also bounded uniformly in $\ee$. We recall~\eqref{e:maxpw} and by applying the Helly-Kolmogorov-Fr\'echet compactness theorem we eventually get the desired pre-compactness result. \\
{\sc Step 2:} fix a sequence $\ee_k \to 0^+$, then owing to {\sc Step 1} and up to subsequences $w_{\ee_k} \to u$ in $L^1_{\mathrm{loc}}(\R_+ \times \R)$ 
for some function $u \in L^\infty (\R_+ \times \R)$. Owing to Theorem~\ref{t:entropy}, $u$ is the entropy admissible solution of~\eqref{e:cl} and by the uniqueness of such solution this yields the first convergence result in~\eqref{e:vali}. \\
{\sc Step 3:}  let $u$ be as in {\sc Step 2}, we now show that $u_\ee \weaks u$ weakly$^\ast$ in $L^\infty (\R^+ \times \R)$.
Owing to~\eqref{e:maxpu}, the family $\{ u_\ee \}$ is pre-compact in $L^\infty (\R^+ \times \R)$ endowed with the weak$^\ast$ topology. To conclude, it suffices to show that any accumulation point $v$ satisfies $v=u$. To this end, we recall that, for any $\varphi \in C^\infty_c (\R_+ \times \R)$ we have the identity 
$$
   \iint_{\R_+ \times \R} \varphi w_\ee  dx dt = 
   \iint_{\R_+ \times \R} \varphi [u_\ee \ast \eta_\ee] dx dt =  \iint_{\R_+ \times \R}[ \varphi \ast \check \eta_\ee ] u_\ee dx dt  \quad \text{provided $\check \eta_\ee (x) : = \eta_\ee (-x)$}.
$$
By passing to the limit in the above inequality and using {\sc Step 2} and the arbitrariness of $\varphi$ we then arrive at $v =u$. }
\subsection{Proof of the convergence rate~\eqref{e:rate}}
We first introduce the so-called Kru{\v{z}}kov's entropy-entropy flux pairs
\be \label{e:ken}
     \alpha_c (u) : = |u-c|, \qquad \beta_c (u) = \mathrm{sign} (u-c) [V(u)u - V(c) c], \qquad c \in \R
\eq
and establish a preliminary result. 
\begin{proposition}\label{p:quasisolutions} {Assume that $u_0$, $V$ and $\eta$ satisfy~\eqref{e:u0}~\eqref{e:V} and~\eqref{e:eta}, respectively, and let  $w_\ee$  be as in~\eqref{e:w}, where $u_\ee$ is the solution of the Cauchy problem~\eqref{e:nlc}. If $\eta(\xi) \xi \in L^1(\R)$} there is a constant $K>0$ which only depends on $V$ and $\eta$ and satisfies the following properties. Fix $\ee>0$, then there is a function 
$E_\ee \in L^1_{\mathrm{loc}}(\R_+ \times \R)$ such that 
\be \label{e:dc}
   D^\ee_c (\phi) := \iint_{\R_+ \times \R}  \big[ \alpha_c(w_{\ee}) \partial_t \phi + \beta_c(w_{\ee}) \partial_x \phi \big] dx dt \ge -    \iint_{\R_+ \times \R}   E_\ee | \partial_ x\phi|  dx dt \text{ }
\eq
for every $c\in \R$ and every test function $\phi \in C^\infty_c(]0, + \infty[ \times \R), \; \phi \ge0$; also, 
\be \label{e:inte}
       \int_\R |E_\ee(t,x)| dx \le K \ee \mathrm{TotVar}\, w_\ee(t, \cdot) \quad
       \text{for a.e. $t \in \R_+$}.
\eq
\end{proposition}
\begin{proof}
We fix a test function $\phi \in C^\infty_c(]0, + \infty[ \times \R), \; \phi \ge0$ and proceed according to the following steps.\\
{\sc Step 1:}  we fix a (classical) entropy-entropy flux pair, i.e. $\alpha, \beta \in C^2 (\R)$ satisfy $\alpha''\ge 0$ and $\beta'(u)=\alpha'(u) [uV'(u) + V(u)]$. We recall~\eqref{e:dalpha} and~\eqref{e:perdopo2} and conclude that 
\be \label{e:altalena}
    D^\ee_\alpha (\phi) \ge - \max_{w \in [0, 1]} |\alpha'(w)|
    \iint_{\R_+ \times \R} | S_\ee (t, x) |  |\partial_x \phi| dx dt
    +\iint_{\R_+ \times \R} u_{\ee_k}(t,y) [ G^1_{\ee} (t, y) +G^2_{\ee} (t, y) ] dy dt,
\eq
where $S_\ee$, $G^1_{\ee} $, $G^2_{\ee}$ are as in~\eqref{e:esse} and \eqref{e:G}, respectively.
We now set $L(u) : = I(u) - \alpha'(u) V(u)$ and, recalling the equality $I'(u) = \alpha''(u)V(u)$, conclude that $L'(u) = - \alpha'(u) V'(u)$. We then point out that 
\[
\begin{split}
(G^1_\ee + G^2_\ee)(t,y) = &~ \int_{-\infty}^y  \big[ L(w_{\ee}(t,y)) -L (w_{\ee}(t,x)) \big]  \frac1{\ee}\eta\left(\frac{x-y}{{\ee}}\right) \partial_x \phi (t,x) dx\\
& + \int_{-\infty}^y    \alpha'(w_{\ee}(t,x))\big[ V(w_\ee(t,y)) - V(w_\ee(t,x)) \big]\frac1{\ee}\eta\left(\frac{x-y}{{\ee}}\right) \partial_x  \phi (t,x)dx,
\end{split}
\]
which owing to Fubini's theorem and by recalling~\eqref{e:maxpu} yields 
\be \label{e:scivolo}
\begin{split}
 & \iint_{\R_+ \times \R}  (G^1_\ee + G^2_\ee)(t,y) u_\ee(t,y) dy dt  \\
\ge & -   \iint_{\R_+ \times \R}  |\partial_x  \phi | (t,x) \int_x^{+\infty} \big|L(w_{\ee}(t,y)) - L (w_{\ee}(t,x)) \big|  \frac1{\ee}\eta\left(\frac{x-y}{{\ee}}\right) dy dx dt \\
&  - \max_{w \in [0, 1]} |\alpha'(w)|  \iint_{\R_+ \times \R}   |\partial_x  \phi| (t,x) \int_x^{+\infty}   \big| V(w_\ee(t,y)) - V(w_\ee(t,x)) \big|\frac1{\ee}\eta\left(\frac{x-y}{{\ee}}\right) dy dx dt \\
& \stackrel{L'= - \alpha'V'}{\ge}\! \!
- 2 \max_{w \in [0, 1]} |\alpha'(w)|  \essup_{w \in ]0, 1[} | V'(w)| \iint_{\R_+ \times \R}   |\partial_x  \phi| (t,x) \int_x^{+\infty}  \! \! \big| w_\ee(t,y) - w_\ee(t,x) \big|\frac1{\ee}\eta\left(\frac{x-y}{{\ee}}\right) dy dx dt.
\end{split}
\eq
{\sc Step 2:} we fix $c \in \R$ and consider the Kru{\v{z}}kov's entropy-entropy flux pairs defined in~\eqref{e:ken}. We construct a sequence $(\alpha_n, \beta_n)$ of $C^2$ 
entropy-entropy flux pairs such that $\alpha_n \to \alpha_c$ and $\beta_n \to \beta_c $ as $n \to + \infty$ in $C^0 (\R)$ and $|\alpha'_n| \leq 1$ for every $n$. By comparing~\eqref{e:dalpha} and~\eqref{e:dc} and recalling that $\phi(0, \cdot) \equiv 0$ we get that $D^\ee_{\alpha_n} (\phi) \to D^\ee_c (\phi)$ as $n \to + \infty$. 
By passing to the limit in~\eqref{e:altalena} and using~\eqref{e:scivolo} we arrive at the inequality
\begin{equation*} \begin{split}
    D^\ee_c (\phi) &   \ge -
     \int_0^T  \! \!\int_\R  \! \! |   S_\ee (t, x) |  |\partial_x \phi| dx dt \\
     & - \iint_{\R_+ \times \R}   |\partial_x  \phi| (t,x) \underbrace{2   \essup_{w \in ]0, 1[} |V'(w)| \int_x^{+\infty}  \! \! \big| w_\ee(t,y) - w_\ee(t,x) \big|\frac1{\ee}\eta\left(\frac{x-y}{{\ee}}\right) dy}_{A_\ee(t, x)} dx dt.
  \end{split}\end{equation*}
We now set $E_\ee: = S_\ee + A_\ee $ and by recalling~\eqref{e:esse} point out that 
\begin{equation*} \label{e:inta} \begin{split}
    \int_\R | S_\ee (t, x) | dx & = \int_\R \int_x^{+\infty}\underbrace{|u_\ee (t, y)|}_{\leq 1\ \text{by \eqref{e:maxpu}}} \big|V(w_{\ee}(t,y)) - V (w_{\ee}(t,x)) \big|  \frac1{\ee}\eta\left(\frac{x-y}{{\ee}}\right) dy dx \\
   &\stackrel{\xi = \frac{x-y}{\ee}}{\leq}
  \int_\R \int_{\R_-} \big|V(w_{\ee}(t,x - \ee \xi )) - V (w_{\ee}(t,x)) \big| \eta (\xi) d\xi dx  \\
   & \stackrel{\text{Fubini}}{=}
   \int_{\R_-} \eta (\xi) \int_{\R} \big|V(w_{\ee}(t,x - \ee \xi )) - V (w_{\ee}(t,x)) \big| dx d\xi 
  \\ &
  \leq \essup_{w \in ]0, 1[} |V'(w) |  \int_{\R_-} \eta (\xi) \int_{\R} \big|w_{\ee}(t,x - \ee \xi ) - w_{\ee}(t,x) \big|  dx d\xi 
 \\& \stackrel{\eqref{e:tvchar}}{\leq}
 \ee   \essup_{w \in ]0, 1[} |V'(w) | \mathrm{TotVar}[  w_\ee] (t, \cdot) \underbrace{\int_{\R_-} \eta(\xi) |\xi| d \xi}_{C_\eta} 
 {\leq} C_\eta \ee \essup_{w\in ]0, 1[}| V'(w)| 
      \mathrm{TotVar}\,  w_\ee(t, \cdot).
\end{split}
\end{equation*}
By relying on the same computations we control the integral of $|A_\ee|$ and eventually arrive at~\eqref{e:inte}.  
\end{proof}
We now establish the proof of~\eqref{e:rate}. We follow an argument due to Kuznetsov~\cite{Kuznetsov}, which in turn relies on the doubling-of-variables technique by Kru{\v{z}}kov~\cite{Kruzkov}. We detail the argument for the sake of completeness. First, we apply Proposition \ref{p:quasisolutions} and recall that $u$ is an entropy admissibile solution of~\eqref{e:cl}; we conclude that for every test function $\phi \in C^\infty (({]0, +\infty[} \times \R)^2)$, $\phi \ge 0$ we have
\begin{equation}\label{e:doubling}
\begin{split}
 &\iint_{\R_+ \times \R}  \big[ |w_\ee(t',x') - u(t,x)| \partial_{t'} \phi + q(w_\ee(t',x'), u(t,x)) \partial_{x'} \phi \big]dx' dt' \ge  -  \iint_{\R_+ \times \R}  E(t',x')|\partial_{x'} \phi|dt'dx' \\
& \iint_{\R_+ \times \R}  \big[|w_\ee(t',x') - u(t,x)|\partial_{t} \phi  + q(w_\ee(t',x'), u(t,x)) \partial_{x} \phi \big]dx dt \ge ~ 0,
\end{split}
\end{equation}
provided $q(a,b) = \mathrm{sign}(a-b) \big( aV(a) - bV(b)\big)$.
We now choose the test function $\phi$ by setting 
\be \label{e:phi}
\phi(t,x,t',x') = \psi \left( \frac{t+t'}2\right) \chi \left( \frac{x+x'}2\right) \gamma_{\nu_1}(t-t') \gamma_{\nu_2}(x-x'),
\eq
 where $\psi \in C^\infty_c(]0, + \infty[)$, $\chi \in C^\infty_c (\R)$ satisfy $\psi,  \chi\ge 0$. Also, $\nu_1,\nu_2>0$ are two parameters and we have used the notation $\gamma_{\nu_i} (x): = \nu_i^{-1} \gamma (\nu_i^{-1}x)$, where $\gamma$ is a standard convolution kernel satisfying 
$$
   \gamma \in C^\infty_c (]-1, 1[), \quad \gamma \ge 0, \quad  \int_\R \gamma =1. 
$$
We plug~\eqref{e:phi} into~\eqref{e:doubling}, integrate the first equation with respect to $dx dt$ and the second equation with respect to $dx'dt'$ and add the resulting equations; we get
\be \label{e:iiii}
\iiiint_{(\R_+\times \R)^2} \! \! \! \big[ |w_\ee - u| \psi'\chi + q(w_\ee, u) \psi \chi' \big]\gamma_{\nu_1} \gamma_{\nu_2} dx' dt' dx dt \ge - \iiiint_{(\R_+\times \R)^2} 
   \! \! \! E^\ee(t',x') |\partial_{x'} \phi| dx' dt' dx dt,
\eq
where we used  the equalities $\partial_t \phi + \partial_{t'}\phi =  \psi'\chi  \ \gamma_{\nu_1} \gamma_{\nu_2}$
and $\partial_x \phi + \partial_{x'}\phi = \psi \chi '\ \gamma_{\nu_1} \gamma_{\nu_2}$. {In the previous expression and in the following if not otherwise specified the functions $\chi$ and $\chi'$ are evaluated at $(x+ x')/2$ and the functions $\gamma_{\nu_1}$ and $\gamma_{\nu_2}$ at $t-t'$ and $x-x'$, respectively. }
To control the right-hand side of~\eqref{e:iiii}, we first of all point out that 
\be \label{e:ddue}
         |\partial_{x'} \phi (t, x, t', x')| \leq \frac{1}{2}\| \psi \chi'  \|_{C^0}  \gamma_{\nu_1} \gamma_{\nu_2} + 
          \| \psi \chi \|_{C^0}  \gamma_{\nu_1} \frac{1}{\nu^2_2}\left|  \gamma' \left(\frac{x-x'}{\nu_2} \right) \right|.
\eq
Next, we choose $t_1$ and $t_2$ such that 
$\mathrm{supp} \ \psi \subseteq ]t_1, t_2[ $ and arrive at 
\be \label{e:stimae}
\begin{split}
\iiiint_{(\R_+\times \R)^2} &E^\ee(t',x') |\partial_{x'} \phi| dx dt dx' dt' \le  ~ \int_{t_1- \nu_1/2}^{t_2+ \nu_1/2} \int_\R E^\ee(t',x') 
          \left(\iint_{\R_+\times \R} |\partial_{x'} \phi| dx dt\right) dx'dt' \\
\stackrel{\eqref{e:ddue}}{\leq} &~  \int_{t_1- \nu_1/2}^{t_2+ \nu_1/2}
     \int_{\R} E^\ee(t',x') \left[\frac{1}{2}\| \psi \chi'  \|_{C^0}  + 
          \| \psi \chi \|_{C^0}  \frac{\| \gamma'\|_{L^1}}{\nu_2} \right] 
       dx'dt' \\
       \stackrel{\eqref{e:inte}}{\leq} &
       K \ee  [t_2 - t_1+\nu_1]  \left[ \frac{1}{2}\| \psi \chi'  \|_{C^0} + 
          \| \psi \chi \|_{C^0}  \frac{\| \gamma'\|_{L^1}}{\nu_2} \right]  \essup_{t \in \R_+} \mathrm{TotVar}\, w_\ee(t, \cdot)\\
           \stackrel{\eqref{e:tv},\eqref{e:tv0}}{\leq} &
         K   \ee  [t_2 - t_1+\nu_1]  \left[ \frac{1}{2}\| \psi \chi'  \|_{C^0}+ 
          \| \psi \chi \|_{C^0}  \frac{\| \gamma'\|_{L^1}}{\nu_2} \right]  \mathrm{TotVar}\, u_0.\end{split}
\eq
We now let $\nu_1 \to 0^+$ and then consider a sequence $\psi_n$ such that 
$$
    \psi_n \weaks  \mathbf 1_{]t_1,t_2[} \;\text{weakly$^\ast$ in $L^\infty (\R_+)$}, \quad 
    \| \psi_n \|_{C^0} \leq 1 
$$
and by taking the $n \to + \infty$ limit in both~\eqref{e:iiii} and~\eqref{e:stimae} and recalling that 
$w_\ee, u \in C^0 (\R_+, L^1_{\mathrm{loc}}(\R))$
we arrive at 
\be \label{e:maggio}
\begin{split}
\iint_{\R^2}& |w_\ee(t_2,x') - u(t_2,x)|\chi \gamma_{\nu_2} dx dx'  - \iint_{\R^2}|w_\ee(t_1,x') - u(t_1,x)|\chi \gamma_{\nu_2} dx dx' \\
 & \le
 \int_{t_1}^{t_2} \iint_{\R^2}  q(w_\ee, u)  \chi'  \gamma_{\nu_2} dx dx' dt 
  +  K \ee  [t_2 - t_1]  
          \left[ \frac{1}{2}\|  \chi'  \|_{C^0}+ 
          \| \chi \|_{C^0}  \frac{\| \gamma'\|_{L^1}}{\nu_2} \right]  \mathrm{TotVar}\, u_0.\end{split}
\eq
Assume $\| \chi \|_{C^0} \leq 1$; then we have
\be \label{e:giugno} \begin{split}
\iint_{\R^2}& |w_\ee(t_2,x') - w_\ee(t_2,x)|\chi \gamma_{\nu_2} dx dx'   \stackrel{\xi= \frac{x-x'}{\nu_2}}{\leq} 
\int_\R \gamma (\xi) \int_\R |w_\ee(t_2,x- \nu_2 \xi) - w_\ee(t_2,x)| dx d \xi \\&
\stackrel{\eqref{e:tvchar}}{\leq}  
C_\gamma \nu_2  \mathrm{TotVar}\, w_\ee(t_2, \cdot) 
   \stackrel{\eqref{e:tv},\eqref{e:tv0}}{\leq} C_\gamma \nu_2  \mathrm{TotVar}\, u_0
\end{split}
\eq
and analogously 
\be \label{e:luglio} \begin{split}
\iint_{\R^2}& |w_\ee(t_1,x') - w_\ee(t_1,x)|\chi \gamma_{\nu_2} dx dx' \leq C_\gamma \nu_2  \mathrm{TotVar}\, u_0.
\end{split}
\eq
We now point out that 
\be \label{e:ottobre}\begin{split}
\int_\R &  |w_\ee(t_2,x) - u(t_2,x)| \int_\R \chi \gamma_{\nu_2}  d x' dx =
\iint_{\R^2}  |w_\ee(t_2,x) - u(t_2,x)|  \chi \gamma_{\nu_2} dx dx' \\
& \leq 
  \iint_{\R^2}|w_\ee(t_2,x) - w_\ee(t_2,x')| \chi \gamma_{\nu_2} dx dx'  +
 \iint_{\R^2}|w_\ee(t_2,x') - u(t_2,x)| \chi \gamma_{\nu_2} dx dx' \\
\end{split}
\eq
and that 
\be \label{e:piscina}\begin{split}
 \iint_{\R^2} |w_\ee(t_1,x') - u(t_1,x)| \chi \gamma_{\nu_2} &dx dx' \leq 
 \iint_{\R^2}|w_\ee(t_1,x') - w_\ee(t_1,x)| \chi \gamma_{\nu_2} dx dx' \\ & +
  \int_\R   |w_\ee(t_1,x) - u(t_1,x)| \int_\R \chi  \gamma_{\nu_2}d x' dx.
  \end{split}
\eq
By plugging the above inequalities into~\eqref{e:maggio} and recalling~\eqref{e:giugno},~\eqref{e:luglio} and the inequality $\chi \leq 1$ we arrive at 
\[ \begin{split}
 \int_\R   &|w_\ee(t_2,x) - u(t_2,x)|  \int_\R \chi  \gamma_{\nu_2}  d x' dx \leq 2 C_\gamma \nu_2  \mathrm{TotVar}\, u_0 \\ 
& + 
\int_\R   |w_\ee(t_1,x) - u(t_1,x)| \int_\R \chi  \gamma_{\nu_2} d x' dx {+}  \int_{t_1}^{t_2} \iint_{\R^2}  q(w_\ee, u)  \chi'  \gamma_{\nu_2} dx dx' dt
 \\ & 
+  K \ee  [t_2 - t_1]  
          \left[ \frac{1}{2}\|  \chi'  \|_{C^0}+ 
          \frac{\| \gamma'\|_{L^1}}{\nu_2} \right]  \mathrm{TotVar}\, u_0 
\end{split}
\]
and by using the inequality
\be \label{e:novembre} \begin{split}
    \int_\R   |w_\ee(0,x) - u_0(x)| \int_\R & \chi  \gamma_{\nu_2}  d x' dx
    \leq  \int_\R   |w_\ee(0,x) -u_0(x)| dx \\ & \stackrel{\eqref{e:w}}{=}  \| u_0 \ast \eta_\ee - u_0 \|_{L^1}   
      \stackrel{\eqref{e:tvl1}}{\leq} C_\eta \ee \mathrm{TotVar}\, u_0
      \end{split}
\eq
and letting $t_1 \to 0^+$ we get 
\be \label{e:settembre} \begin{split}
 \int_\R   &|w_\ee(t_2,x) - u(t_2,x)|  \int_\R \chi  \gamma_{\nu_2}  d x' dx \leq 2 C_\gamma \nu_2  \mathrm{TotVar}\, u_0 \\ 
& + C_\eta \ee \mathrm{TotVar}\, u_0 + \int_{0}^{t_2} \iint_{\R^2}  q(w_\ee, u)  \chi'  \gamma_{\nu_2} dx dx' dt
+  K \ee  t_2   
          \left[ \frac{1}{2}\|  \chi'  \|_{C^0}+ 
          \frac{\| \gamma'\|_{L^1}}{\nu_2} \right]  \mathrm{TotVar}\, u_0.
\end{split}
\eq
Assume for a moment we have shown that 
\be \label{e:agosto}
      \int_\R   |w_\ee(t,x) - u(t,x)|  dx < + \infty \; \text{for every $t \in \R_+$};
\eq
then we can consider a sequence of test functions $\{ \chi_n \}$ such that $\chi_n \weaks 1$ weakly$^\ast$ in $L^\infty (\R)$, $0 \leq \chi_n \leq 1$ and $\| \chi'_n \|_{C^0} \to 0^+$ as $n \to +\infty$. By passing to the $n \to + \infty$ limit in~\eqref{e:settembre} we arrive at 
\[ \begin{split}
 \int_\R   &|w_\ee(t_2,x) - u(t_2,x)| dx \leq 2 C_\gamma \nu_2  \mathrm{TotVar}\, u_0 + C_\eta \ee \mathrm{TotVar}\, u_0 + K \ee  t_2  
                   \frac{\| \gamma'\|_{L^1}}{\nu_2}  \mathrm{TotVar}\, u_0 
\end{split}
\]
and by choosing $\nu_2 = \sqrt{\ee t_2}$ and by relying on the arbitrariness of $t_2$ we get~\eqref{e:rate}. We are thus left to establish~\eqref{e:agosto}: to this end, we recall that $q(a,b) = \mathrm{sign}(a-b) \big( aV(a) - bV(b))$, which implies that $|q(w_\ee, u) |\leq \essup_{w \in ]0, 1[} \{ |V(w)| + |V'(w)|\} |w_\ee (t, x) - u (t, x')|$. We can then choose a test function\footnote{Note that, technically speaking, the only compactly supported function satisfying the inequality $|\chi'| \leq \chi$ is the constant $0$. However, one can use an easy approximation argument and show that the inequality~\eqref{e:settembre} holds for instance for any function in the Schwartz space $\mathcal S (\R)$.} $\chi$ such that $|\chi'| \leq \chi$ and conclude that 
$$
    \int_{0}^{t_2} \iint_{\R^2}  q(w_\ee, u)  \chi'  \gamma_{\nu_2} dx dx' dt \leq \essup_{w \in ]0, 1[} \{ |V(w)| + |V'(w)|\}   \int_{0}^{t_2} \iint_{\R^2}  |w_\ee (t, x) - u (t, x')|  \chi \gamma_{\nu_2} dx dx' dt. 
$$
By controlling the above term with the same argument as in~\eqref{e:piscina}, applying Gr\"onwall's lemma and recalling~\eqref{e:novembre} we establish a bound on the left-hand side of~\eqref{e:settembre} independent of $\chi$, and this eventually implies~\eqref{e:agosto} by the arbitrariness of $t_2$.   
\section{Proof of Theorem~\ref{t:cex}} \label{s:cex}
In \S\ref{ss:out} we provide an overview of the construction of the counter-example, and we describe the basic ideas involved. Next, in \S\ref{ss:cexp},~\S\ref{ss:cexb} and \S\ref{ss:cexf} we detail the construction. 
Throughout this section we assume $V(w) = 1-w$, $\eta = \mathbbm{1}_{]-1, 0[}$, which implies that~\eqref{e:w} and~\eqref{e:wx} boil down to 
\be \label{e:wwx}
     w_\ee(t, x): = \frac{1}{\ee} \int_{x}^{x+ \ee} u(t, y) dy \quad \text{and} \quad      \frac{\partial w_\ee}{\partial x} (t, x) = \frac{u_\ee (t, x+\ee) - u_\ee(t, x)}{\ee},
\eq
respectively. Note that, since both $V$ and $\eta$ are now fixed, the construction of the counter-example boils down to the construction of the initial datum $u_0$. 
\subsection{Outline}\label{ss:out}
Our construction is reminiscent of the one in~\cite{ColomboCrippaMarconiSpinolo}, in particular we assume that $u_0 (x) =1$ for a.e. $x>0$, which implies that $u_\ee (t, x) =1$ for a.e. $(t, x)$ such that $x>0$. Also, as in~\cite{ColomboCrippaMarconiSpinolo} the total variation increase mechanism is triggered by a countable number of ``building blocks" located on the negative real axis. The building block we use here, however, is different from the one in~\cite{ColomboCrippaMarconiSpinolo}, and capturing the total variation increase mechanism requires a finer analysis. We now provide an handwaving description of the main ideas involved, and we refer to the following paragraphs for the rigorous argument. 

The basic building block is represented in Figure~\ref{f} and consists of two rectangles of height $h$ and length $\ell$, located at distance $2 \ell$ and $6 \ell$ from the origin, respectively.  
\begin{figure}
\begin{center}
\caption{In red the building block $v_{h \ell}$ triggering the total variation increase for $\ee= 4 \ell$ } 
\psfrag{a}{$\ee= 4 \ell$} 
\psfrag{b}{$-2 \ell$}
\label{f}
\includegraphics[scale=0.6]{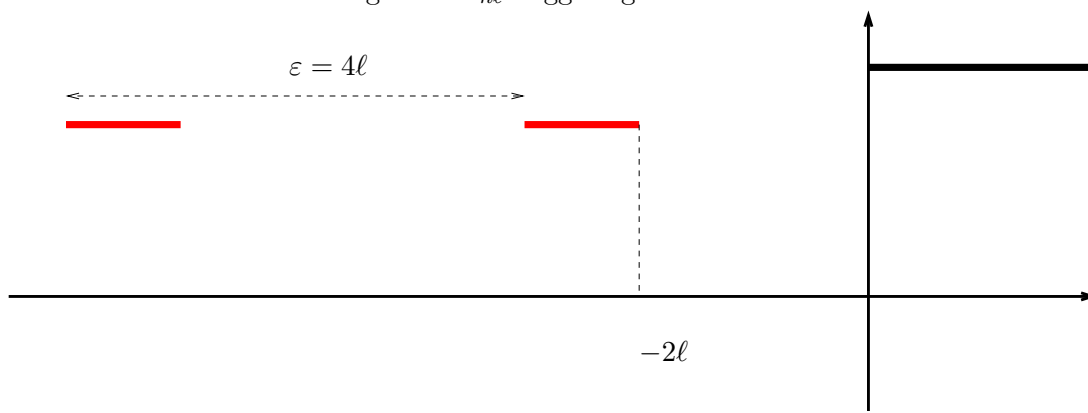}
\end{center}
\end{figure}
In other words, 
\be \label{e:bblock}
    v_{h \ell} (x) : = 
    \left\{
    \begin{array}{ll}
              0 & x <  - 7 \ell  \\
              h & - 7 \ell  < x <  - 6 \ell \\
              0  &  -6 \ell  < x < - 3 \ell \\
              h &  - 3 \ell < x < - 2 \ell \\
              0 & x > - 2 \ell \\
    \end{array}
   \right. \qquad \ell \in ]0, 1[, \; h \in ]0, 1[. 
\eq
We now fix $\ee >0$, set $u_0(x) = v_{h \ell} (x) + \mathbbm{1}_{]0, + \infty[}(x)$ and consider the Cauchy problem~\eqref{e:nlc}; let $w_\ee$ be as in~\eqref{e:w}, then it is fairly easy to see that
\be \label{e:limits}
     \lim_{x \to - \infty} w_\ee (t, x) =0, \quad \lim_{x \to + \infty} w_\ee (t, x) = 1 \quad \text{for every $t \ge 0$ and every $\ee>0$.}
\eq
Assume that $\ee = 4 \ell$; by using formula~\eqref{e:wwx} we realize (see equation~\eqref{e:vasu}) that $w_\ee (0, \cdot)$ is a monotone non-decreasing function and hence $\mathrm{Tot Var} \;
w_\ee (0, \cdot) =1$ by~\eqref{e:limits}. Assume for a moment that we have shown that, for some time $t>0$, we have $\partial_x w_\ee (t, \cdot) <0$ on some interval, then by using~\eqref{e:limits} we infer $\mathrm{Tot Var} \;w_\ee (t, \cdot)~>~1$, that is we have an increase of the total variation. 

To establish the inequality $\partial_x w_\ee (t, \cdot) <0$  we rely on~\eqref{e:wwx} and on the method of characteristics. More precisely, in the following we term $X_\ee(t, s, \xi)$ the solution of the Cauchy problem
\begin{equation}
\label{e:ch}
        \left\{
        \begin{array}{ll}
                     \displaystyle{\frac{dX_\ee}{dt}= 1- w_\ee (t, X_\ee)} \\ 
                     \displaystyle{X_\ee (s, s, \xi) = \xi, \phantom{\int}}
       \end{array}
        \right.
\end{equation}
in other words $X_\ee (\cdot, s, \xi)$ is the characteristic line that attains the value $\xi$ at time $t=s$. Note that the equation at the first line of~\eqref{e:nlc} implies that, if $u_0(x)=0$  then $u_\ee (t, X_\ee (t, 0, x))=0$ for every $t$, that is $u_\ee$ vanishes along the whole characteristic starting at $x$. Also, if  $u_0(y)>0$  then $u_\ee (t, X_\ee (t, 0, y))>0$ for every $t$. 
In particular this yields that if $x \in ]-2\ell, 0[$ and $y \in ]-7 \ell, -6 \ell[$ then  $u_\ee (t, X_\ee (t, 0, x))=0$ and $u_\ee (t, X_\ee (t, 0, y))>0$ for every $t$, respectively.   
The key observation is now that, since one can show that $w_\ee (0, - 6 \ell) < w_\ee (0, -2 \ell)$ then the initial speed of the characteristic $X_\ee (\cdot, 0, -6 \ell)$ is \emph{greater} than the initial speed of the characteristic $X_\ee (\cdot, 0, -2 \ell)$. Let us now consider a point $y$ just at the left of $- 6 \ell$, then $u_0(y) = h>0$ and hence $u_\ee (t, X_\ee (t, 0, y))>0$ for every $t$. On the other hand, let us consider the point $X_\ee (t, 0, y) + \ee$ and recall that $\ee = 4 \ell$: since the characteristic line $X_\ee (\cdot, 0, -2 \ell)$ is slower than $X_\ee (\cdot, 0, y)$, then the value $X_\ee (t, 0, y) + \ee$ overtakes $X_\ee (\cdot, 0, -2 \ell)$, and this in turn implies that $u_\ee (t, X_\ee (t, 0, y) + \ee)=0$. Summing up, we have $u_\ee (t, X_\ee (t, 0, y))>0$ and $u_\ee (t, X_\ee (t, 0, y) + \ee)=0$, and owing to~\eqref{e:wwx} this yields $\partial_x w_\ee (t, X_\ee (t, 0, y))<0$. 

To complete the construction we fix a sequence $\ee_n$, construct a corresponding sequence of building blocks $v_{h_n \ell_n}$ with $\ee_n = 4 \ell_n$, and juxtapose them on the negative real axis, in other words
\be \label{e:idcex}
    u_0(x) : = \sum_{n=1}^\infty v_{h_n \ell_n} (x) + \mathbbm{1}_{]0, + \infty[} (x), \quad \ell_n : = \ee_n /4.
\eq
The main obstruction we have now to tackle is that the previous considerations show that the building block $v_{h_n \ell_n}$ triggers a total variation increase at scale $\ee_n = 4 \ell_n$, but these considerations fail when $\ee_n \neq 4 \ell_n$. Hence, at a given scale $\ee_n$ it may in principle happen that the total variation \emph{increase} given by building block $v_{h_n \ee_n/4}$ is smaller than the total variation \emph{decrease} triggered  by the other building blocks in $u_0$, and that the overall effect is that the total variation decreases. To rule out this possibility, we  have to rely on a finer analysis and make the previous qualitative considerations quantitative. 
\subsection{Preliminary results} \label{ss:cexp}
We recall~\eqref{e:ch}. Note that, by the equation at the first line of~\eqref{e:nlc},  we have 
\begin{equation} \label{e:cmass}
    \int_{X_\ee (t, s, \xi_1)}^{X_\ee (t, s, \xi_2)} u_\ee (t, y) dy = 
    \int_{\xi_1}^{\xi_2} u_\ee(s, y) dy \quad \text{for every $s, t \in \R$, $\xi_1 < \xi_2$}, 
\end{equation}
that is the mass between two characteristic lines is always conserved. We now recall some basic properties established in~\cite{ColomboCrippaMarconiSpinolo}.
\begin{lemma}\label{l:old}
Fix $\ee >0$ and assume $u_0 \in L^1_\mathrm{loc}(\R)$ satisfies $0 \leq u_0 \leq 1$ and $u_0(x) =1$ for a.e. $x>0$. Then $u_\ee (t, x)=1$ for every $t\ge 0$ and a.e. $x>0$ and 
\begin{equation} \label{e:stali}
           X_\ee (t, 0, 0) =0 \quad \text{for every $t \ge 0$.}
\end{equation} 
Also, 
\begin{equation}
\label{e:adestra}
        \frac{dX_\ee}{dt}(t, s, \xi) \ge 0, \quad \text{for every $\xi \in \R$, $t, s \in \R$}. 
\end{equation}
\end{lemma}
We also have 
\begin{lemma}
Under the same assumptions as in Lemma~\ref{l:old}, assume furthermore that there is $x_0 <0$ such that $u_0(x)=0$ for a.e. $x< x_0$. Then for every $\ee>0$ we have 
\begin{equation}
\label{e:pippo}
        w_\ee (t, x) =0 \; \text{for every $t >0$, $x< x_0 -\ee$}, \qquad  w_\ee (t, x) = 1
        \; \text{for every $t>0$, $x>0$}.
\end{equation}
\end{lemma}
\begin{proof}
Owing to~\eqref{e:adestra}, $u_\ee (t, x) =0$ for a.e. $x<x_0$, $t \ge 0$ and $\ee >0$. This yields the first equation in~\eqref{e:pippo}.  The second equation follows from the equality $u_\ee (t, x) =1$ for every $t \ge 0$ and a.e. $x >0$. 
\end{proof}
Note that~\eqref{e:pippo} yields~\eqref{e:limits}. 
\subsection{Analysis of the basic building block} \label{ss:cexb}
The following two lemmas deal with the case $\ell = 4 \ee$.
\begin{lemma}
\label{l:marzo}
Set 
\begin{equation}\label{e:u0cex}
    u_0 (x) =  v_{h \ell} (x) + s(x) + \mathbbm{1}_{]0, + \infty[} (x),
\end{equation}
where $v_{h \ell}$ is the same as in~\eqref{e:bblock} and $s \in L^\infty (\R)$ satisfies 
\begin{equation}
\label{e:esse2}      
       0 \leq s \leq 1, \qquad \text{$s(x) =0$ for a.e. $x \notin ]- \delta, 0[$}, \; \delta< 2 \ell. 
\end{equation} 
If  $\ee = 4 \ell$, then $\mathrm{Tot Var} \,  w_\ee (0, \cdot) =1$. 
\end{lemma}
\begin{proof}
We recall the formula 
\begin{equation}\label{e:formula}
     \mathrm{Tot Var} \,  w_\ee (t, \cdot) = \lim_{x \to + \infty} w_\ee (t, x) -  \lim_{x \to - \infty} w_\ee (t, x)
     + 2 \int_\R \left[ \frac{\partial w_\ee}{\partial x}  (t, x) \right]^- dx \stackrel{\eqref{e:limits}}{=} 
     1   + 2 \int_\R \left[ \frac{\partial w_\ee}{\partial x}  (t, x) \right]^- dx,  
\end{equation}
 where $[\cdot]^-$ denotes the negative part. Next, we point out that
\be \label{e:vasu}
    \frac{\partial w_\ee (0, x) }{\partial x}   = 
    \left\{ 
    \begin{array}{ll}
     0 &  x <  - 7 \ell - \ee  \\
     \displaystyle{\frac{h }{ \ee}  }&  - 7 \ell - \ee < x < - 7 \ell   \\
     0 & -7 \ell  < x < - 6 \ell \\
    \displaystyle{\frac{u_0 ( x+ \ee) }{ \ee}  }&  - 6 \ell < x < - 4 \ell \\
      \displaystyle{\frac{1 - u_0( x) }{ \ee} } & - 4 \ell  < x < 0 \\
    0 & x >0,
    \end{array}
  \right. 
\eq
which yields $\partial w_\ee (0, \cdot) / \partial x \ge 0$ and hence concludes the proof owing to~\eqref{e:formula}. 
\end{proof}
We now quantify the total variation increase triggered by the initial datum in~\eqref{e:u0cex}. 
\begin{lemma}
\label{l:maggio}
Under the same assumptions as in Lemma~\ref{l:marzo}, we have 
\begin{equation}
\label{e:aumenta}
        \mathrm{TotVar} \, w_\ee (t, \cdot) \ge 1 + 2 \left(\frac{2 - h}{4} \right) h  t + o(t) \quad t  \to 0^+.
\end{equation} 
\end{lemma}
\begin{proof}
We proceed according to the following steps. \\
{\sc Step 1:} we focus on $X_\ee (t, 0, - 2 \ell)$. Since 
$$
    w_\ee (0, - 2 \ell) \stackrel{\eqref{e:wwx},\eqref{e:u0cex}}{\ge }\frac{1}{\ee} \int_{ 0  }^{- 2 \ell + \ee} 1\ dy =
    1 - \frac{2 \ell}{\ee} \stackrel{\eqref{e:ch}}{\implies} \left. \frac{d X_\ee (t, 0, - 2 \ell)}{d t} \right|_{t=0}\leq   \frac{2 \ell}{\ee},
$$
then  
\begin{equation} \label{e:asinistra}
     X_\ee (t, 0, - 2\ell) \leq - 2 \ell + \frac{2 \ell}{\ee} t + o(t)
    \stackrel{\ee = 4 \ell}{=} 
    - 2 \ell + \frac{1}{2} t + o(t) \quad t \to 0^+. 
\end{equation}
{\sc Step 2:} we focus on $X_\ee (t, 0, - 6 \ell)$. Note that  
$$
    w_\ee (0, - 6 \ell) \stackrel{\eqref{e:bblock}}{=}\frac{h \ell}{\ee} \implies \left. \frac{d X_\ee (t, 0, - 6 \ell)}{d t} \right|_{t=0}= 1- \frac{h \ell}{\ee}, 
$$
and by recalling that $\ee = 4 \ell$ this yields 
\begin{equation} \label{e:luglio2}
    X_\ee (t, 0, - 6 \ell) + \ee = - 2 \ell + \left( 1 - \frac{h}{4} \right) t + o(t)\quad t \to 0^+. 
\end{equation}
{\sc Step 3:}  by comparing~\eqref{e:asinistra} and~\eqref{e:luglio2} we get that  
$
     X_\ee (t, 0, - 6 \ell) + \ee > X_\ee (t, 0, - 2 \ell)
$ for $t>0$ small enough and this in turn implies that 
$$
    X_\ee (0, t, X_\ee (t, 0, - 2\ell) - \ee)) \in ]- 7 \ell, - 6 \ell [
$$
for $t$ sufficiently small. Since for every $y \in  ]- 7 \ell, - 6 \ell [$ we have 
$$
    w_\ee (0, \cdot)=\frac{h \ell}{\ee}  \stackrel{\ee= 4 \ell}{=} 
     \frac{h }{4} \implies
   \left. \frac{d X_\ee (t, 0, y)}{d t} \right|_{t=0}= 1- \frac{h}{4} 
    \implies X_\ee (t, 0, y) = y+ \left( 1 - \frac{h }{4} \right) t + o(t)\quad t \to 0^+   
$$
then owing to~\eqref{e:asinistra} we get
\begin{equation} \label{e:ottobre2}
    X_\ee (0, t, X_\ee (t, 0, - 2\ell) - \ee) \leq  - 6\ell  +  \left( - \frac{1}{2}  + \frac{h }{4} \right)  t + o(t) 
    =   - 6\ell  +  \left( \frac{-2 + h}{4} \right)  t + o(t)  
\quad t \to 0^+. 
\end{equation}
{\sc Step 4:} since  $2\ell - \delta > 0$ by assumption, then $X_\ee (t, 0, -2 \ell)< 
X_\ee (t, 0, -6 \ell) + \ee < 
X_\ee (t, 0, -\delta)$ for sufficiently small $t>0$. On the other hand, $u_\ee (t, x) =0$ for every $x  \in ]X_\ee (t, 0, - 2 \ell), X_\ee (t, 0, - \delta)[$. \\
{\sc Step 5:} we point out that, for every $\xi \in ]-7 \ell, - 6 \ell[$ we have $u_\ee (t, X_\ee(t, 0, \xi))>0$. On the other hand, if $X_\ee (t, 0, \xi) + \ee \in ]X_\ee (t, 0, -2 \ell), X_\ee (t, 0, -\delta)[$, then by {\sc Step 4} we have $u_\ee (t, X_\ee (t, 0, \xi) + \ee) =0$ and hence 
\begin{equation} \label{e:lunedi}
    \frac{\partial w_\ee}{\partial x} (t, X_\ee(t, 0, \xi)) = 
   \frac{u_\ee (t, X_\ee (t, 0, \xi) + \ee) - u_\ee (t, X_\ee (t, 0, \xi) ) }{\ee} =
   -  \frac{u_\ee (t, X_\ee (t, 0, \xi) ) }{\ee} < 0. 
\end{equation}
{\sc Step 6:} by combining the previous steps we have that, if $\xi \in ] X_\ee (0, t, X_\ee (t, 0, -2 \ell) - \ee), - 6 \ell [$, then~\eqref{e:lunedi} holds true. This implies that 
\begin{equation*} \begin{split}
    \int_\R \left[ \frac{\partial w_\ee}{\partial x}  (t, x) \right]^- dx & \ge 
    - \int_{X_\ee (t, 0, -2\ell) -\ee}^{ X_\ee (t, 0, - 6 \ell)}\frac{\partial w_\ee}{\partial x}  (t, x) dx \stackrel{\eqref{e:lunedi}}{=
}    \frac{1}{\ee} \int_{X_\ee (t, 0, -2\ell) -\ee}^{ X_\ee (t, 0, - 6 \ell)} u_\ee (t, y) dy \\ &
    \stackrel{\eqref{e:cmass}}{=}
    \frac{1}{\ee} \int_{ X_\ee (0, t, X_\ee (t, 0, -2 \ell) - \ee)}^{ - 6 \ell} u_0 (t, y) dy  
     \stackrel{\eqref{e:ottobre2}}{\ge}
   \left(\frac{2- h}{4} \right) ht + o(t) \quad t\to 0^+
\end{split}
\end{equation*}
and by recalling~\eqref{e:formula} this yields~\eqref{e:aumenta}. 
\end{proof}
In the following lemma we consider the case $ \ell > \max\{\ee + \delta, 2 \ee\}$. 
\begin{lemma} \label{l:bfrombelow}
Assume $ \ell > \max\{\ee + \delta, 2 \ee\}$ and that $u_0$ is given by~\eqref{e:u0cex} with $s$ satisfying~\eqref{e:esse2}; then 
\begin{equation} \label{e:tevere}
    \int_{- 7 \ell - \ee }^{- 2 \ell}  \left| \frac{\partial w_\ee}{\partial x} (0, x)\right| dx  =
     4 h 
\end{equation}
and 
\begin{equation} \label{e:marzo}
   \int_{X_\ee (t, 0, - 7 \ell) - \ee}^{X_\ee (t, 0, - 2 \ell)} \left| \frac{\partial w_\ee}{\partial x} (t, x)\right|dx \ge  4 h   + o(t) \quad \text{as $t \to 0^+$}. 
\end{equation}
\end{lemma} 
\begin{proof}
By using the inequality $\ell- \delta > \ee$ we get  
$$
     \frac{\partial w_\ee}{\partial x} (0, x) = \frac{u_0 (x + \ee) - u_0 (x)}{\ee} =
    \left\{
     \begin{array}{ll}
     h/\ee& - 7 \ell - \ee< x < - 7 \ell \\
     0 & - 7 \ell < x < - 6 \ell - \ee \\
     -  h/\ee &  - 6 \ell - \ee < x < - 6 \ell \\ 
     0 & - 6 \ell < x < - 3 \ell - \ee \\
     h/\ee  & - 3 \ell - \ee < x < - 3 \ell \\
    0 & - 3 \ell < x < - 2 \ell - \ee \\
     - h/\ee & - 2 \ell - \ee < x < - 2 \ell  \\
    \end{array}
   \right. \qquad w_\ee (0, -7 \ell - \ee)=0= w_\ee (0, - 2\ell)
$$
and this yields~\eqref{e:tevere}.  We are left to establish~\eqref{e:marzo}. We point out that 
$$
    X_\ee (t, 0, -6\ell) + \ee < X_\ee (t, 0, - 3 \ell), \quad 
     X_\ee (t, 0, -2\ell) + \ee <   X_\ee (t, 0, - \delta)
$$
provided $t$ is sufficiently small. The above inequalities imply 
$$
    w_\ee (t, X_\ee (t, 0, -6\ell) ) = 0=w_\ee (t,  X_\ee (t, 0, -2\ell)) 
$$
and, since we also have  $w_\ee (t, X_\ee (t, 0, - 7 \ell) - \ee)=0= w_\ee (t, X_\ee (t, 0, - 3\ell) - \ee)$ for $t$ sufficiently small, this in turn yields 
\begin{equation}
    \label{e:danubio}
\begin{split}
 \int_{X_\ee (t, 0, - 7 \ell) - \ee}^{X_\ee (t, 0, - 2 \ell)}  \left| \frac{\partial w_\ee}{\partial x} (t, x)\right|dx & \ge 
      2\sup \{ w_\ee (t,  x), \; x \in ] X_\ee (t, 0, - 7 \ell) -\ee,  X_\ee (t, 0, - 6 \ell) [ \} \\
      & \quad + 
     2\sup \{ w_\ee (t,  x), \; x \in ] X_\ee (t, 0, - 3\ell) - \ee,  X_\ee (t, 0, - 2 \ell) [ \}  \\
     & \ge 2 w_\ee (t, X_\ee (t, 0, - 7 \ell) + 2   w_\ee (t, X_\ee (t, 0, - 3 \ell) \phantom{\int}
\end{split}
\end{equation}
for $t$ sufficiently small. To evaluate $w_\ee (t, X_\ee (t, 0, - 7 \ell))$ we point out that, since $\eta = \mathbbm{1}_{]-1, 0[}$ and $V(w) = 1-w$, the equation~\eqref{e:ewgen} for the material derivative of $w_\ee$ boils down to 
$$
      \partial_t w_\ee + [1-w_\ee] \partial_x w_\ee = u_\ee (t, x+ \ee) \frac{w_\ee (t, x+ \ee)-w_\ee (t, x) }{\ee}.
$$
Since $\ell > 2 \ee$, then $w_\ee(0, -7 \ell) = h=w_\ee (0, - 7 \ell + \ee)$ and hence by the above formula 
$$
   \left. \frac{d w_\ee (t,  X_\ee (t, 0, - 7 \ell))}{dt} \right|_{t=0} = 0 \implies 
   w_\ee (t,  X_\ee (t, 0, - 7 \ell)) = h + o(t) \quad t \to 0^+. 
$$
By analogous considerations we get $ w_\ee (t,  X_\ee (t, 0, - 3 \ell)) = h + o(t) \quad t \to 0^+$ and by plugging these equalities into~\eqref{e:danubio} we arrive at~\eqref{e:marzo}. 
\end{proof}
\subsection{Conclusion of the proof of Theorem~\ref{t:cex}} \label{ss:cexf}
We fix a sequence $\{ \ee_n \}$ satisfying~\eqref{e:ticino} and we take the same intial datum $u_0$ as in~\eqref{e:idcex}, where $\{h_n\}$ is any sequence such that $0 \leq h_n \leq 1$ and the series $\sum_{n=1}^\infty h_n$ converges. We now show that, for any $n \in \mathbb N$, we can find $t_n>0$ satisfying~\eqref{e:gennaio}. \\
{\sc Step 1:} if $n=1$ we set $s(x) : = \sum_{k=2}^\infty v_{h_k \ell_k} (x),$
which satisfies~\eqref{e:esse2} provided $\delta : =  8 \ell_2 < 2 \ell_1$. Since $\ell_2 = \ee_2/4$, $\ell_1= \ee_1/4$, this inequality boils down to $2 \ee_2 < \ee_1/2$, which is satisfied owing to~\eqref{e:ticino}. 
By combining Lemma~\ref{l:marzo} and Lemma~\ref{l:maggio} we then get that~\eqref{e:gennaio} holds true
for $n=1$ and some $t_1$ sufficiently small.  \\
{\sc Step 2:} we now fix $n>1$ and evaluate $\mathrm{Tot Var} \, w_{\ee_n} (t, \cdot)$ for $t=0$ and for $t \to 0^+$.  
The basic idea to do so is that we separately consider the contribution to the total variation of each of the first $n-1$ building blocks and then the contribution of all the remaining ones. It turns out that to each of the first $n-1$ building blocks  we can apply Lemma~\ref{l:bfrombelow} and hence conclude that for each of them  the total variation can decrease at most of $o(t)$. We can apply Lemma~\ref{l:marzo} and Lemma~\ref{l:maggio} to the remaining  blocks and conclude that their total variation increases of some factor proportional to $t$. By adding the two contributions, we arrive at~\eqref{e:gennaio}. We now provide the technical details. \\
{\sc Step 2A:} we show that each of the first $n-1$ evolve independently of the rest of the solution, at least for sufficiently small $t>0$. To this end, we point out that, for every $k=1, \dots, n-1$,
\be \label{e:potenzedidue}
    8 \ell_{k+1} + \ee_n = 2 \ee_{k+1} + \ee_n \stackrel{\eqref{e:ticino}}{<}
    \frac{1}{8} \ee_{k} + \frac{1}{16}  \ee_{k} = \frac{3}{16} \ee_{k} < 
    \frac{1}{4} \ee_k =  \ell_k  < 2 \ell_k.
\eq
This implies that $ - 7 \ell_{k+1} - \ee_n > - 2 \ell_k $ and hence that 
$X_{\ee_n} (t, 0, - 7 \ell_{k+1}) -\ee_n > X_{\ee_n} (t, 0, - 2 \ell_k) $ provided $t$ is sufficiently small. In particular, for any such $t$ we have 
\begin{equation}
\label{e:reno}
        \mathrm{Tot Var} \, w_{\ee_n} (t, \cdot) =
        \sum_{k=1}^{n-1}
        \int_{X_{\ee_n} (t, 0, -7 \ell_k) - \ee_n}^{ X_{\ee_n} (t, 0, - 2 \ell_k) }
    \left| \frac{\partial w_{\ee_n}}{\partial x}  (t, x) \right| dx 
    + \int_{X_{\ee_n} (t, 0, - 7 \ell_n) -\ee_n}^0 \left| \frac{\partial w_{\ee_n}}{\partial x}  (t, x) \right| dx.
\end{equation}
{\sc Step 2B:} we now fix $k=1, \dots, n-1$ and evaluate the $k$-th term in the sum in~\eqref{e:reno}. We set 
$
    s(x) : = \sum_{j=k+1}^\infty v_{h_j, \ell_j} (x)$
and $\delta : =  8 \ell_{k+1} $. Owing to~\eqref{e:potenzedidue} and to~\eqref{e:ticino} we have $\ell_k > \max\{ 2 \ee_n, \ee_n + \delta \}$. By repeating the same argument as in the proof of Lemma~\ref{l:bfrombelow} we then get 
\begin{equation}
\label{e:rodano}
          \int_{ -7 \ell_k - \ee_n}^{  - 2 \ell_k }
    \left| \frac{\partial w_{\ee_n}}{\partial x}  (0, x) \right| dx = 4 h_k, \qquad 
      \int_{X_{\ee_n} (t, 0, -7 \ell_k) - \ee_n}^{ X_{\ee_n} (t, 0, - 2 \ell_k) }
    \left| \frac{\partial w_{\ee_n}}{\partial x}  (t, x) \right| dx \ge 4h_k + o(t) \quad t \to 0^+. 
\end{equation}
{\sc Step 2C:} we control the second term in~\eqref{e:reno}. We set $
    s(x) : = \sum_{j=n+1}^\infty v_{h_j \ell_j} (x)$
and $\delta : =  8 \ell_{n+1} $ and point out that $\delta < 2 \ell_n$ by~\eqref{e:ticino}. 
By applying Lemma~\ref{l:marzo} and Lemma~\ref{l:maggio} we get 
\be \label{e:rio}
    \int_{ - 7 \ell_n -\ee_n}^0 \left| \frac{\partial w_{\ee_n}}{\partial x}  (0, x) \right| dx=1
    \qquad 
     \int_{X_{\ee_n} (t, 0, - 7 \ell_n) -\ee_n}^0 \left| \frac{\partial w_{\ee_n}}{\partial x}  (t, x) \right| dx=
   1 +  2 \left(\frac{2 - h_n}{4} \right) h_n  t  + o(t) \quad t  \to 0^+.
\eq
{\sc Step 2D:} we plug~\eqref{e:rodano} and~\eqref{e:rio} into~\eqref{e:reno} and conclude that
 $   \mathrm{Tot Var} \, w_{\ee_n} (0, \cdot) =4 \sum_{k=1}^{n-1} h_k +1$ and that 
$$   \mathrm{Tot Var} \, w_{\ee_n} (t, \cdot) \ge 4\sum_{k=1}^{n-1} h_k +1 +  2 \left(\frac{2 - h_n}{4} \right) h_n t + o(t) \quad t  \to 0^+,
$$
which establishes~\eqref{e:gennaio}.

\section*{Acknowledgments}
M.C. and E.M. are supported by the SNF Grant 182565. G.C. is supported by the ERC StG 676675 FLIRT. 
L.V.S. is a member of the PRIN 2020 Project 20204NT8W4 and of the GNAMPA  group of INDAM. 
\bibliographystyle{plain}
\bibliography{tv}
\end{document}